\documentclass{article}
\usepackage[12pt]{extsizes}
\usepackage[utf8]{inputenc}
\usepackage[english]{babel}
\usepackage{amsthm, amssymb, amsmath, amsfonts, amscd}
\usepackage{epigraph}
\ifx\pdfoutput\undefined
\usepackage{graphicx}
\else
\usepackage[pdftex]{graphicx}
\fi
\usepackage{hyperref}
\usepackage{ulem}
\usepackage{textcomp}
\usepackage{euscript}
\usepackage{multirow}
\usepackage[space]{grffile}
\usepackage{wrapfig}
\usepackage[labelsep=none]{caption}
\usepackage{color}
\usepackage{verbatim}

\numberwithin{equation}{section}
\addto\captionsrussian{}
\addto\captionsrussian{}
\addto\captionsrussian{}
\addto\captionsrussian{}

\newtheorem{innercustomthm}{Theorem}
\newenvironment{customthm}[1]
  {\renewcommand\theinnercustomthm{#1}\innercustomthm}
  {\endinnercustomthm}

\theoremstyle{definition}

\theoremstyle{plain}
\newtheorem{thr}{Theorem}

\newtheorem{cor}{Corollary}
\newtheorem{lemma}{Lemma}
\newtheorem{propos}{Proposition}

\theoremstyle{remark}
\newtheorem{rem}{Remark}

\newcommand{\RN}{\ensuremath\mathbb{R}} 
\newcommand{\NN}{\ensuremath\mathbb{N}} 
\newcommand{\ZN}{\ensuremath\mathbb{Z}} 

\newcommand{\eps}{\ensuremath\varepsilon}

\newcommand{\sht}{^{\prime}}

\newcommand{\intlm}{\int\limits}
\newcommand{\ib}{\intlm_{\Omega_R}}
\newcommand{\nbu}{\nabla u_R}
\newcommand{\ns}{\nabla \sigma}
\newcommand{\n}{\mathfrak{n}}

\newcommand{\supp}{\mathrm{supp}}
\newcommand{\emb}{\hookrightarrow}

\begin{document}

\title{Abundance of entire solutions \\to nonlinear elliptic equations\\by the variational method}
\author{L.M.\,Lerman$^1$, P.E.\,Naryshkin$^2$, A.I.\,Nazarov$^{2, 3}$\\
\normalsize $^1$ Lobachevsky State University of Nizhny Novgorod\\
\normalsize
\normalsize$^2$ Saint Petersburg State University \\
\normalsize$^3$ St. Petersburg Dept. of Steklov Math. Institute
}
\date{}
\maketitle
\begin{abstract}
We study entire bounded solutions to the equation $\Delta u - u + u^3 = 0$ in $\RN^2$.
Our approach is purely variational and is based on concentration arguments
and symmetry considerations. This method allows us to construct in a unified way several
types of solutions with various symmetries (radial, breather type, rectangular,
triangular, hexagonal, etc.), both positive and sign-changing. The method is also applicable for
more general equations in any dimension.
\end{abstract}

\section{Introduction}

Studying solutions to nonlinear elliptic equations is the classical
problem arising both in PDE research and applications of PDEs in
geometry, physics and material sciences. Indeed, many problems of
nonlinear physics can be mathematically formulated as studying steady
states of nonlinear scalar fields in nonlinear media. Thus, it is necessary
to investigate solutions with nontrivial
spatial structures (patterns) to nonlinear elliptic PDEs. One of the
simplest PDEs of this kind is
\begin{equation}
\Delta u+f(u)=0,
\label{0.1}
\end{equation}
considered in some domain $D\subset \mathbb R^n$ together with some boundary
conditions, or in the whole space $\mathbb R^n$ with prescribed behavior at infinity.
Important solutions with patterns for this equation are those which are
localized in one or more variables. This kind of problems
can be linked with the study of self-localized solutions in wave-guide optical channels \cite{CGT},
investigations of ``particle-like'' states of nonlinear fields in some
models of elementary particles \cite{R}, the description of bi-phase separation in fluids \cite{CH, Cahn}
and ordering in binary alloys \cite{AC}. The problem of studying patterns in such equations is,
in a sense, an analogue of problems from the theory of dynamical systems,
where ideally one needs to study and classify all solutions to a given system of autonomous
differential equations. Unfortunately, nowadays the theory of elliptic
PDEs is not as developed as the theory of dynamical systems, where we know at least primary objects
 to be studied (equilibria, periodic orbits, quasi-periodic orbits,
homoclinic and heteroclinic structures etc.). In the absence of a
general theory studying certain interesting but more or less ``simple''
model equations is one of the main problems now. Undoubtedly, equation
(\ref{0.1}) belongs to such models.

Solutions to the equation or the system of equations of type (\ref{0.1}) also
correspond to stationary solutions of the evolution equations such as reaction-diffusion
parabolic or wave PDEs. Stationary solutions for these evolution equations are of the primary
interest since this is a starting point to study non-stationary solutions close to stationary,
various stability problems, asymptotic behavior of solutions, etc. Moreover, as it was mentioned in
\cite{Wei}, PDEs of type (\ref{0.1}) do have solutions with interesting patterns,
and the structure of their solution sets has remained mostly a mystery.
This is especially true for solutions given on the whole space (so-called \textit{entire solutions})
considered in our paper.\

Various methods were developed to find solutions of equations (\ref{0.1}) with some prescribed
structures. In particular, for solutions localized in space one method is to search for
radial solutions, i.e. those which depend only on the radial variable $r$.
It is extremely interesting that the only possible positive bounded solutions of (\ref{0.1})
are radial, if some restrictions on nonlinearity are imposed \cite{BNN}.
When searching for radial solutions the problem is reduced to the study of some specific
solutions of the related nonautonomous second order ODE with ``time'' $r$. For some types of
nonlinearities the existence of infinitely many (sign-changing) radial solutions was proved
in \cite{Jones}. Those results were extended to more complicated higher order equations
\cite{KLSh,Sand_Cal}.

Another method was proposed in \cite{Kirsch} and developed by several authors for elliptic
equations (systems) with nonlinearities of various types, see, for instance,
\cite{Mielke,KS,Sand,BI,BIS}. This method has its roots in the theory of center manifolds for ODEs
and allows one to construct solutions of elliptic equations in cylinder-like domains
that are unbounded in some distinguished direction. In that case the original elliptic PDE can be formally
written as an evolutionary
equation in the proper functional space but this equation is usually ill-posed in the
whole space. However, if one is lucky to find a finite dimensional
invariant center submanifold of the proper smoothness, then the restriction of the
evolutionary equation to this submanifold gives a finite dimensional
flow. Orbits of this flow generate solutions of the original equation.

A different approach to finding solutions is related to bifurcation methods, in
particular, the Lyapunov-Schmidt reduction. These methods allow for construction of two-dimensional
solutions with triangle symmetry \cite{paapeW} and solutions of the equation
(\ref{0.1}) in $\RN^n$ periodic in one variable and decaying in other variables \cite{Dancer}.

In this paper we construct entire bounded solutions to (\ref{0.1}) which have various types of
symmetries and may also decay in some directions. We use purely variational approach
which allows us to construct these solutions using the concentration-compactness principle
\cite{Lio} and symmetry considerations \cite{Ka}.

The paper is organized as follows. We start by studying the model equation in $\RN^2$
\begin{equation}
\Delta u - u + u^3  = 0.
\label{MainEq}
\end{equation}
In Section 2 we are concerned with periodic solutions. In Section 3 we discuss radially symmetric
solutions. Generalization of our results for higher dimensions and more general equations
are presented in Section 4. The basic known facts used in the proofs as well as some necessary
technical lemmata are collected in Appendix \hyperref[ssec:App A]{A}.

In Appendix \hyperref[ssec:App B]{B} we give for the comparison a brief explanation of using
the central manifold method for our equation \eqref{MainEq}. This method was proposed in
\cite{AEKLS} in order to construct the so-called \textit{breather type} solutions which are
localized in one variable and periodic in another variable, see also \cite{AA} and \cite{EKLU}.

In Appendix \hyperref[ssec:App C]{C} we show another approach of construction of breather type
solutions using the Bubnov-Galerkin approximations. This method reduces the problem to
finding homoclinic solutions to some saddle equilibrium of a proper Hamiltonian system. However,
in this way we obtain only an approximate solution of the original problem. See, in this
respect, the paper \cite{Alessio} where the close connection between two-dimensional solutions
of the Allen-Cahn equation stabilizing in one variable at infinity and
heteroclinic solutions of the related one-dimensional equation was established.

We wish to stress that the variational method used in our paper is rather general, applicable
in any dimension and allows us to construct in a unified way several types of solutions
(radial, breather type, rectangular, triangular, hexagonal, etc.), both positive and sign-changing.

\medskip

Let us introduce some notations. We use letter $C$ to denote various positive constants.
To indicate that some $C$ depends on a parameter $a$, we write sometimes $C(a)$.

For a domain $\Omega \subset \RN^n$ we denote by $\Omega_R$ the set $\{Rx \mid x \in \Omega\}$.

We use the notation $B(x, R)$ for a ball of radius $R$ centered at $x$.

Notations $o_R(1)$ and $o_\eps(1)$ mean $o(1)$ as $R \to \infty$ and $\eps \to 0$ respectively.

For $1 < p < \infty$ we define
$$
p^* = \begin{cases}
\frac{np}{n-p}, & p < n; \\
+\infty, & \mbox{otherwise}.
\end{cases}
$$
Recall that $p^*$ is the critical Sobolev embedding exponent when $p < n$.

\section{Periodic solutions in $\RN^2$}

\subsection{Some general results}
\label{ssec:General_results}

Suppose $\Omega$ is a domain in $\RN^2$ with piecewise smooth boundary. Consider the variational
problem for the energy functional

\begin{equation}\label{VarProblem}
J[u] = \frac{\ib (|\nabla u|^2+ |u|^2)\, dx}{\Big(\ib |u|^4 \, dx \Big)^{1/2}}\to\,\min .
\end{equation}
Since $J[cu] = J[u]$ it is equivalent to finding a constrained minimum for the problem
\begin{equation}\label{ConMinProblem}
\ib (|\nabla u|^2+ |u|^2)\, dx \to\,\min; \qquad \ib |u|^4\, dx = 1.
\end{equation}
Suppose this minimum is attained (this holds, for instance, in bounded domains, see Remark \ref{ResForCompact} below). Let $v_R$ be a minimizer. Then it satisfies the Euler-Lagrange
equation as well as the natural boundary conditions. Thus it is a weak solution of the problem
\begin{equation*}
-\Delta v +v = \lambda v^3 \ \mbox{in } \Omega_R; \qquad \left.\frac{\partial v}
{\partial \n}\right|_{\partial \Omega_R} = 0,
\end{equation*}
where $\n$ stands for the unit exterior normal vector on $\partial
\Omega_R$.  Notice that the Lagrange multiplier $\lambda$ coincides with the sought-for
minimum of the problem \eqref{ConMinProblem}. By the elliptic regularity theory,
 $v_R$ is a classical solution.

Since $J[|u|] = J[u]$ we can presume $v_R \ge 0$, and after that the maximum principle gives
$v_R > 0$ in $\Omega_R$. We multiply $v_R$ by $\sqrt{\lambda}$ and get another minimizer
of \eqref{VarProblem} $u_R$ which satisfies
\begin{equation}
-\Delta u +u = u^3 \ \mbox{in } \Omega_R; \qquad \left.\frac{\partial u}{\partial \n}
\right|_{\partial \Omega_R} = 0 .
\label{EulerEq}
\end{equation}
We note that $||u_R||_{L^4} = \sqrt{\lambda}$ evidently depends on $R$. We now prove it cannot
be either too small or too large.

\begin{lemma}
\label{LambdaBoundSepar}
$||u_R||_{L^4} = \sqrt{\lambda(R)}$ is bounded and separated from zero as $R \to \infty$.
\end{lemma}

\begin{proof}
First we observe that for $u$ compactly supported $J[u]$ does not depend on $R$. This implies
$\lambda(R)$ is bounded since it is the minimum for problem \eqref{VarProblem}.

Further, let $\Pi_R : W_2^1(\Omega_R) \to W_2^1(\RN^2)$ stand for the extension operator
(see \cite[Chapter 6, Section 3]{St}). The sequence of norms $||\Pi_R||$ is bounded by some
constant as $R \to \infty$. Since the embedding $W_2^1(\RN^2) \emb L^4(\RN^2)$ is bounded
(this follows, for instance, from \eqref{Lem1ineq}), we obtain
$$||u_R||_{W_2^1(\Omega_R)} \ge C||\Pi_R(u_R)||_{W_2^1(\RN^2)} \ge C||\Pi_R(u_R)||_{L^4(\RN^2)}
\ge C||u_R||_{L^4(\Omega_R)},$$
so $\lambda = J[u_R] = ||u_R||^2_{W_2^1(\Omega_R)}/||u_R||^2_{L^4(\Omega_R)} \ge C$ and we are done.
\end{proof}

\begin{rem}
\label{ResForCompact}
The functional in the problem \eqref{ConMinProblem} is coercive and weakly lower semi-continuous
in $W_2^1(\Omega_R)$. If $\Omega$ is bounded the embedding
$W_2^1(\Omega_R)\emb L_4(\Omega_R)$ is compact so the set of functions satisfying the condition
in the problem \eqref{ConMinProblem} is weakly closed in $W_2^1(\Omega_R)$.
This implies (see \cite[Theorems 24.11 and 26.8]{FK}) that the minimum in \eqref{ConMinProblem}
is attained and the reasoning above can be applied.
\end{rem}

We construct first the simplest families of solutions in $\RN^2$ which have a rectangular and
triangular lattices of periods.

\subsection{Solutions with rectangular symmetry}
\label{ssec:Rectangle}
Suppose $\Omega_R$ is a rectangle $(0, R)\times (0, aR)$ where $a$ is a given number. By Remark
\ref{ResForCompact} we obtain a positive solution $u_R$ to the problem \eqref{EulerEq}
that minimizes the functional \eqref{VarProblem}.

By the Concentration Theorem \ref{Concentration_thr} in Subsection \ref{ssec:Conc_thm}, $u_R$
has exactly one concentration sequence $x_R$ as $R \to \infty$. We consider three
possibilities listed below.
\begin{enumerate}
\item The distance between concentration sequence and the sets of vertices of respective rectangles
is bounded.

\item There is a subsequence with unbounded distance from vertices but bounded distance from
$\partial \Omega_R$.

\item There is a subsequence with unbounded distance from $\partial \Omega_R$.
\end{enumerate}
By Remark \ref{BoundShiftRem} (Appendix \hyperref[ssec:App A]{A}) we can assume $x_R$ is a sequence
of vertices in the first case and $x_R \in \partial\Omega_R$ in the second case.
By Remark \ref{RhoDecreaseRem} (Appendix \hyperref[ssec:App A]{A}) in the case 2 we can choose
$\rho^{\prime}(R) \to \infty$ such that $\rho^{\prime}$ is smaller than the distance
from $x_R$ to the vertices. Similarly, in the case 3 we can assume $\rho^{\prime} <
\mathrm{dist}\left(x_R, \partial \Omega_R\right)$. For $R$ large enough the intersection
$B(x_R, \rho^{\prime}(R)) \cap \Omega_R$ is a quarter-disk in the first case, half-disk in
the second and full disk in the third.

Now we claim that the case 3 is impossible. Indeed, choose $\eps>0$. Let $\sigma$ be a cut-off
function from Lemma \ref{cut_off_lemma} in Subsection \ref{ssec:general_theory}. We
consider the component $h_R$ of $\sigma u_R$ which contains $B(x_R, \rho)$. By Lemma
\ref{cut_off_lemma} $J[h_R] \le J[u_R] + o_\eps(1)$ for $R$ large enough.

Now let  $\tilde{h}_R$ be the symmetric rearrangement of $h_R$. It is well known that $||h_R||_
{W_2^1} \ge ||\tilde{h}_R||_{W_2^1}$ and $||h_R||_{L^4} = ||\tilde{h}_R||_{L^4}$
(see \cite[Section II.9, Corollary 2.35]{Ka}). Therefore, $J[h_R] \ge J[\tilde{h}_R]$. We then
consider a trial function $v_R$ which is a quarter-disk of $\tilde{h}_R$ placed
in the corner of the rectangle.

\begin{figure}[h]
\begin{minipage}[t]{0.49\textwidth}
\includegraphics[scale=1.5]{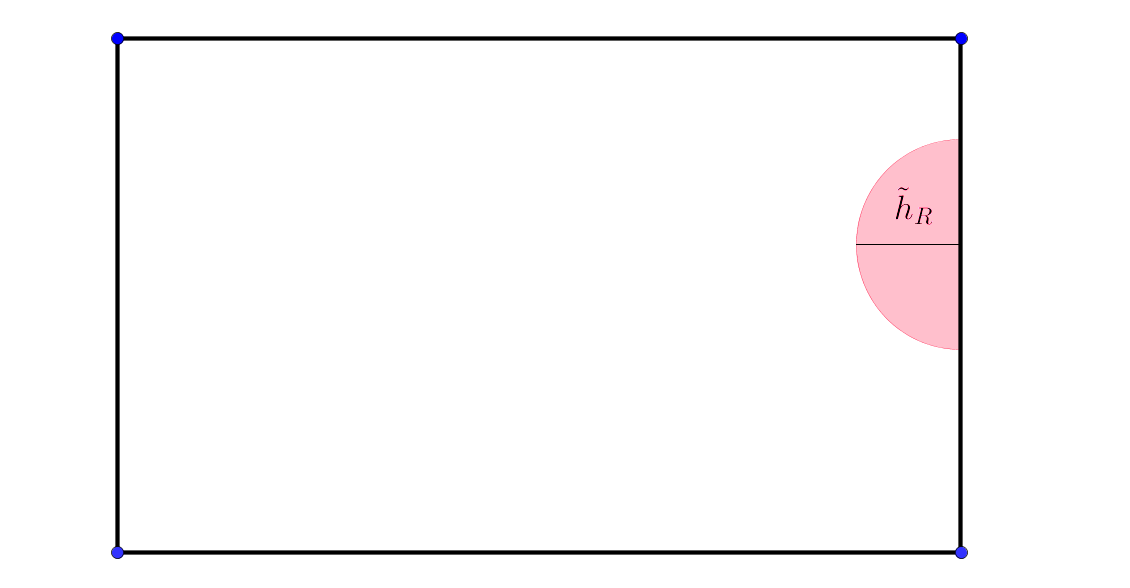}
\end{minipage}
\hfill
\begin{minipage}[t]{0.49\textwidth}
\includegraphics[scale=1.5]{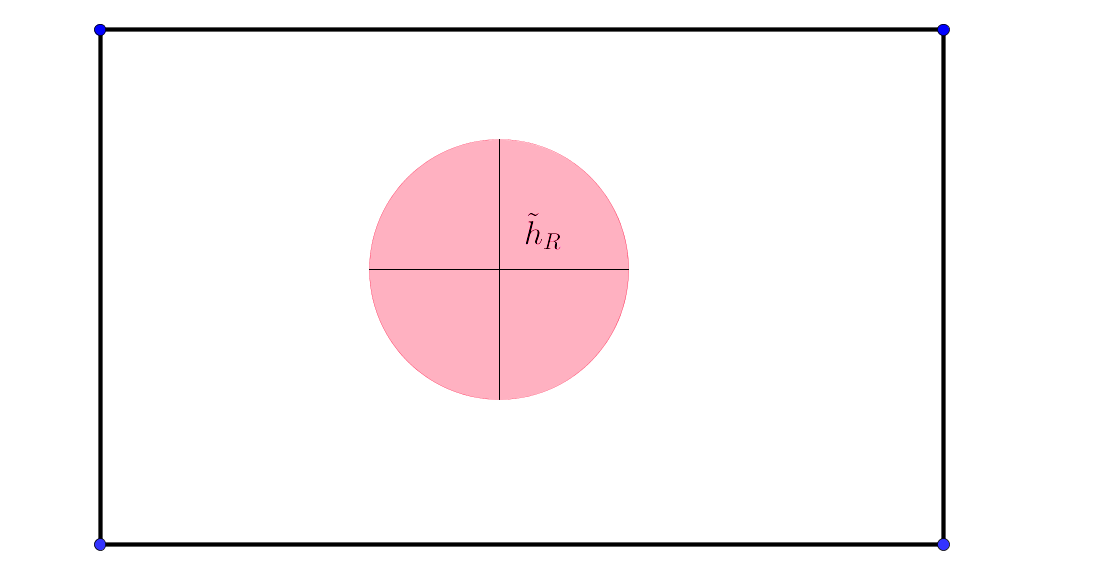}
\end{minipage}
\end{figure}

\includegraphics[scale=2.5]{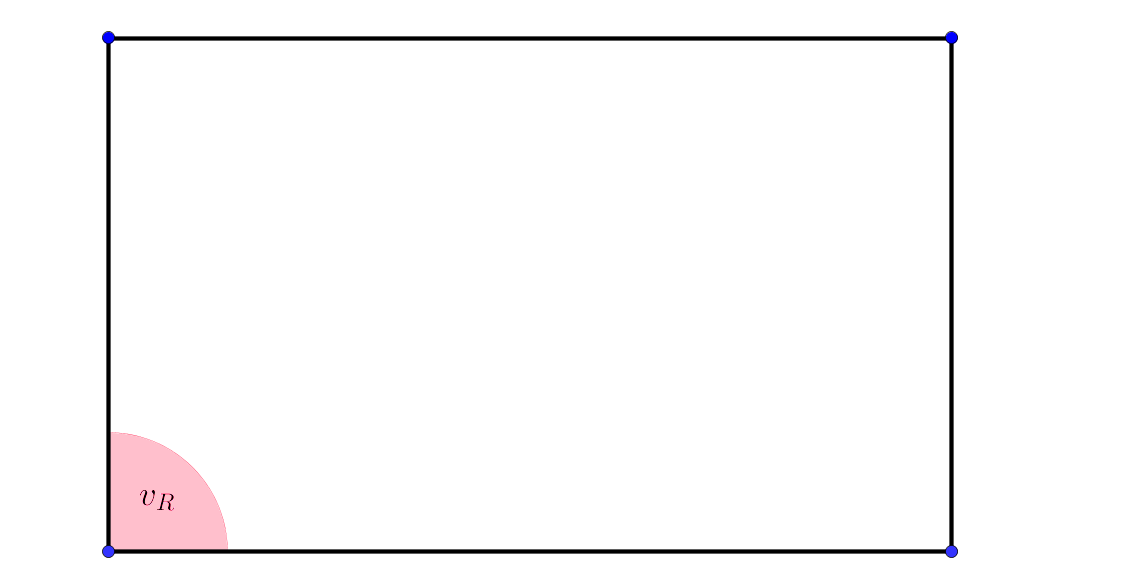}

It is easy to see $J[\tilde{h}_R] = 2J[v_R]$. Therefore, $J[v_R] \le J[u_R]/2 + o_\eps(1)$. This contradicts the minimality of $J[u_R]$ if $R$ is large enough and $\eps$ small enough,
and the claim follows.

The case 2 is likewise impossible.

We conclude that $u_R$ is concentrated in the corner.

Using even reflection, we extend $u_R$ to the rectangle
$(-R,R)\times(-aR,aR)$ and then extend it to $\RN^2$ periodically. Thus, we obtain solutions of \eqref{MainEq} in $\RN^2$ with $(2R, 2aR)$ rectangular lattice of periods
(see Figure~\ref{rectangular lattice}).

Therefore, for every $a \in \RN_+$ for sufficiently large $R$ our periodic solutions are non-trivial and distinct.

\begin{figure}[h]
\begin{minipage}[b]{0.49\textwidth}
\centering
\includegraphics[width=6cm]{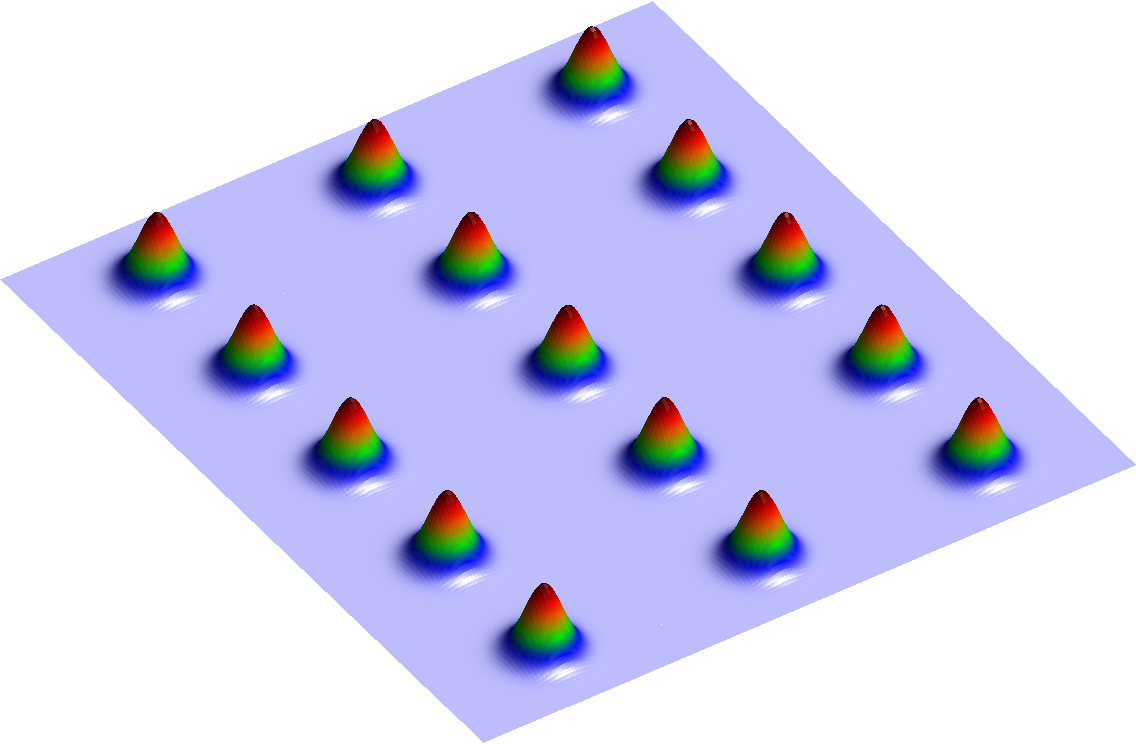}
\end{minipage}
\hfill
\begin{minipage}[b]{0.49\textwidth}
\centering
\includegraphics[width=5cm]{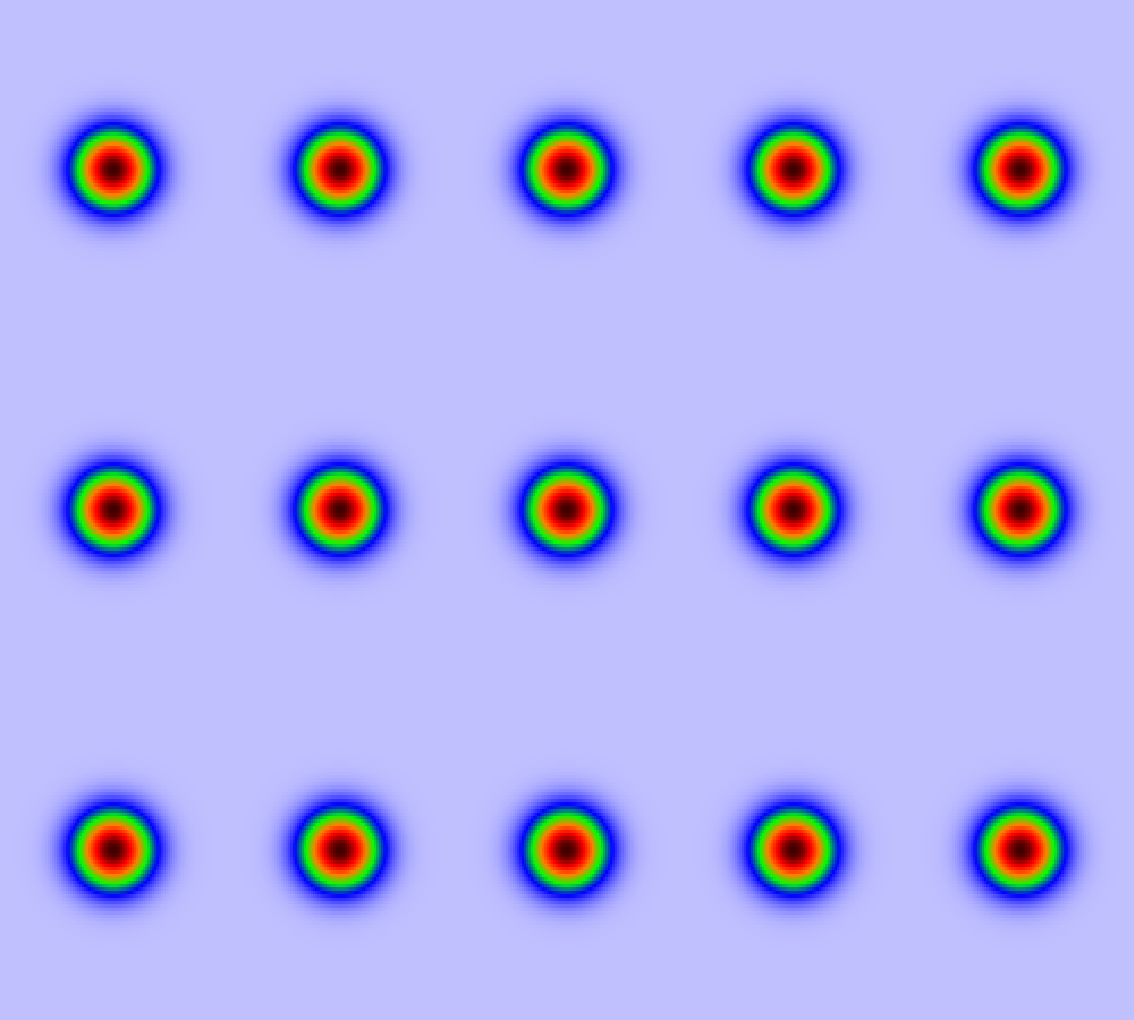}
\end{minipage}
\caption{}
\label{rectangular lattice}
\end{figure}

\begin{cor}
Given $a\ge 1$ and $N\in\NN$ there exists $R^*(a,N) \in \RN$ such that for $R>R^*(a,N)$ the equation \eqref{MainEq} in $\RN^2$ has
at least $N$ different nontrivial positive $(2R,2aR)$-periodic solutions.
\end{cor}
\begin{proof}
We proved that there exists $R_0$ such that $u_R$ is concentrated in the corner for $R > R_0$. For $R>2R_0$ we can consider the problem \eqref{VarProblem}
in $\Omega_R=(0,R)\times(0,aR)$ and in $\Omega_R=(0,R/2)\times(0,aR/2)$.
By the previous argument this gives different solutions, both are nontrivial, positive and $(2R,2aR)$-periodic.
The case of arbitrary $N$ is managed similarly.
\end{proof}

\subsection{Solutions with triangular symmetry}
\label{ssec:Triangle}
Suppose now $\Omega$ is the equilateral triangle with sides of length $1$. By the same argument as in Subsection \hyperref[ssec:Rectangle]{2.2} the minimizer $u_R$
of the variational problem \eqref{VarProblem} is a strong solution of \eqref{EulerEq} and is concentrated in a corner of triangle for $R$ sufficiently large.

\includegraphics[scale=2.5]{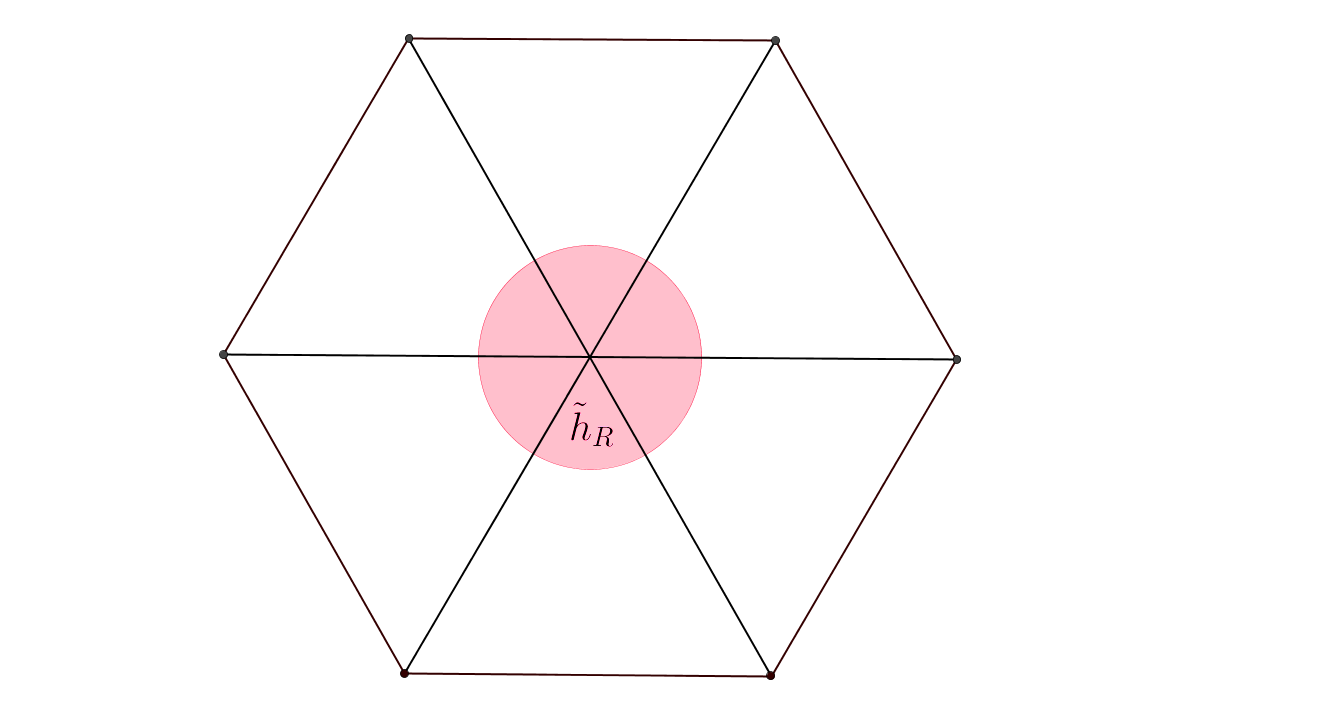}

Using even reflection we extend $u_R$ to the hexagon with side $R$ and then extend it periodically to $\RN^2$. We obtain a solution of \eqref{MainEq} with triangular
lattice with side $2R$ (see Figure~\ref{triangular lattice}).

\begin{figure}[h]
\begin{minipage}[b]{0.49\textwidth}
\centering
\includegraphics[width=6cm]{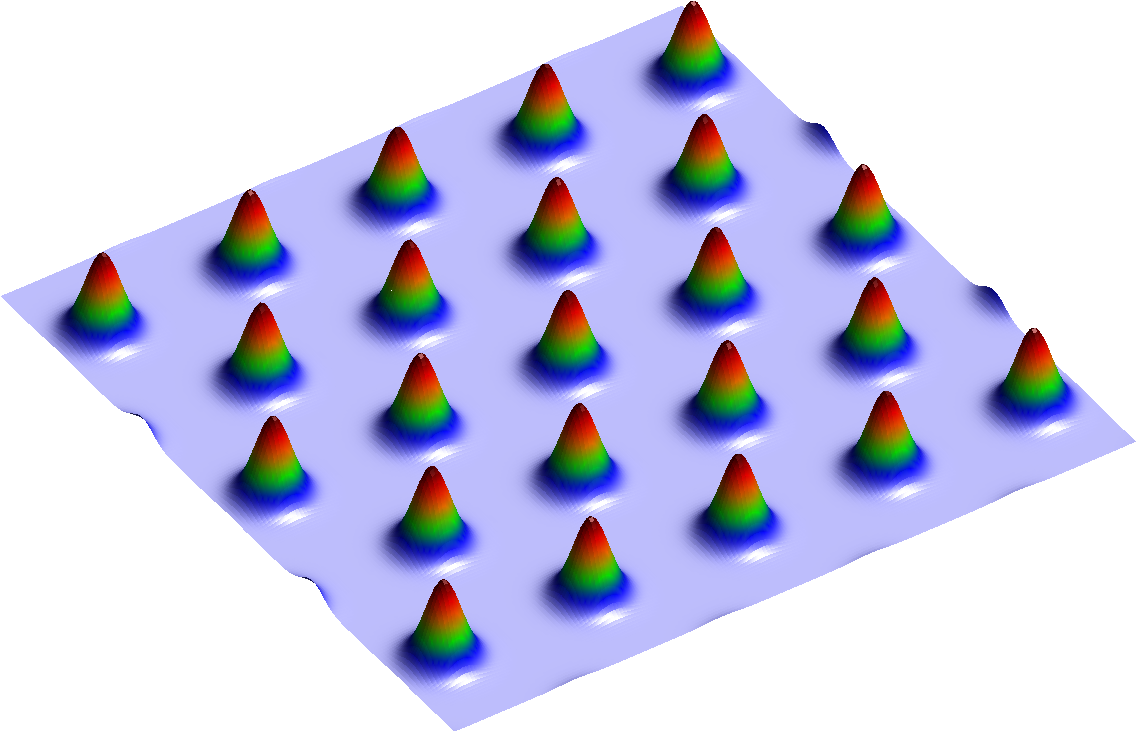}
\end{minipage}
\hfill
\begin{minipage}[b]{0.49\textwidth}
\centering
\includegraphics[width=5cm]{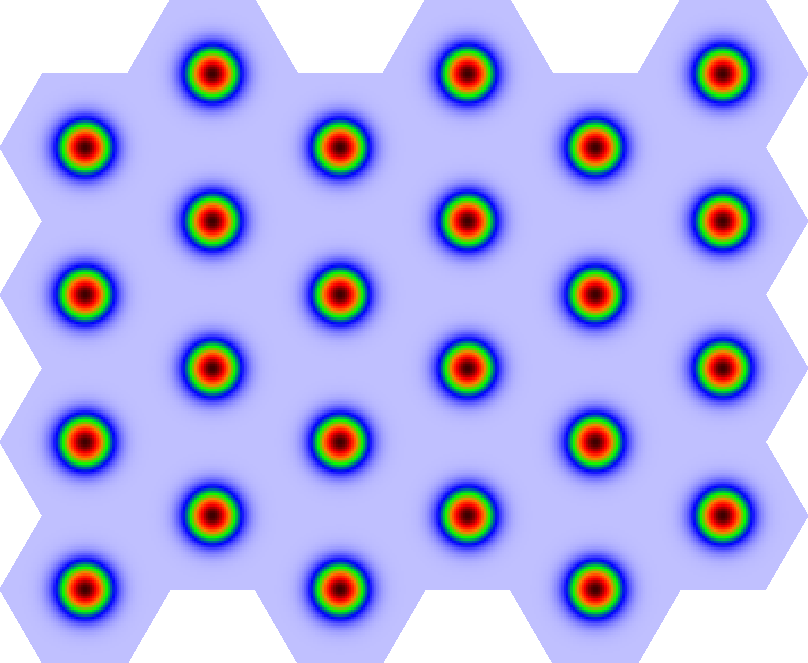}
\end{minipage}
\caption{}
\label{triangular lattice}
\end{figure}

Exactly as in the previous case the Concentration Theorem \ref{Concentration_thr} gives us the following corollary.

\begin{cor}
Given $N \in \NN$ there exists $R^*(N) \in \RN$ such that for $R>R^*(N)$ the equation \eqref{MainEq} in $\RN^2$ has
at least $N$ different nontrivial positive $R$-triangular-periodic solutions.
\end{cor}

\subsection{Solutions with hexagonal symmetry}
For the case of hexagonal lattice of periods a little more intricate argument is required.
Let $\Omega$ be a triangle with angles $\pi/2$, $\pi/3$, $\pi/6$ and hypotenuse of length $1$.
Denote $\Omega_R$ by $XYZ$ where $\angle X = \pi/3$, $\angle Y = \pi/6$, $\angle Z = \pi/2$ and
denote sector $\Omega_R \cap B(Y, R/2)$ by $A_R$. If we consider the usual problem
\eqref{ConMinProblem} in $\Omega_R$ the minimizer is concentrated near $Y$ as it is the corner
with the least angle. After extending it to $\RN^2$ we get a solution with a triangular
lattice of periods similar to Subsection \hyperref[ssec:Triangle]{2.3}. To prevent this we study
the variational problem \eqref{ConMinProblem} with an additional condition
\begin{equation}
\label{HexagonCondition}
\int\limits_{A_R}|u|^4\, dx \le 1/4.
\end{equation}

Note that we can still apply some of the reasoning from Subsection
\hyperref[ssec:General_results]{2.1} and deduce that the minimum is attained and the solution
is non-negative. We now prove a variant of the Concentration Theorem for this case.

\begin{customthm}{2'}
\label{conc_theorem_hexagonal}
Let $u_R$ be a sequence of minimizers for the problem \eqref{ConMinProblem} with the additional
condition \eqref{HexagonCondition}. Then $u_R$ has exactly one concentration sequence of
weight $1$ concentrated near $X$ (see Figure~\ref{Fig3}).
\end{customthm}

\begin{proof}
Suppose there are two concentration sequences $x_R$ and $y_R$. We use Remark \ref{cut-off-rem} (Appendix \hyperref[ssec:App A]{A}) to construct a cut-off function which isolates them.
Let $\sigma_1$ and $\sigma_2$ be the components of $\sigma$ with $x_R \in \supp\,\sigma_1$ and $y_R\in\supp\,\sigma_2$ respectively. Assume first that both of the components $\sigma_1$
and $\sigma_2$ lie within $A_R$. In that case functions $v_R$ constructed in the proof of Lemma \ref{less_than_two_seq_lemma} still satisfy condition \eqref{HexagonCondition} which
contradicts the minimality of $u_R$. The same argument can be applied if both components lie in $\Omega_R \setminus A_R$. Further, suppose the support of one of the components, say,
$\sigma_1$, intersects $(\partial A_R) \cap \Omega_R$. By the construction of $\sigma$ the diameter of $\sigma_1$ is not more than $2(5\rho + \rho^{\prime}(R))/6$. The width of the
annulus around $\sigma_1$ where $\sigma$ equals zero is $4(\rho^{\prime}(R) - \rho)/6$. Therefore, for $R$ large enough we can move the ''bubble'' $\sigma_1u_R$ fully into
$\Omega_R \setminus A_R$ which eliminates this case. Thus, we can presume $\supp\, \sigma_1 \subset \Omega_R \setminus A_R$ and $\supp\, \sigma_2 \subset A_R$.

The proof of the Concentration Theorem \ref{Concentration_thr} shows that the combined weight of the concentration sequences is $1$. After that, the reasoning from Subsection
\hyperref[ssec:Rectangle]{2.2} applied for each sequence shows that we can assume $x_R = X$ and $y_R = Y$. Next we are going to show that it is more profitable to have all of
the weight concentrated near $X$.

Remark \ref{RhoDecreaseRem} allows us to take radii $\rho$ and $\rho^{\prime}(R)$ equal for $x_R$ and $y_R$. Using symmetrization if needed we can assume that
$\sigma_1u_R(x) = h_1(|x_R-x|)$, $\sigma_2u_R(x) = h_2(|y_R-x|)$ where $h_1(t)$ and $h_2(t)$ are decreasing functions vanishing for $t > \rho^{\prime}(R)$.

Consider the function
$$
g(x) = \frac{||\sigma_1u_R||_{L_4}}{||\sigma_2u_R||_{L_4}}\cdot\Big(\frac{\pi/6}{\pi/3}\Big)^{1/4}\cdot h_2(|x_R-x|).
$$
It is easy to see that $||g||_{L_4} = ||\sigma_1u_R||_{L_4}$. Therefore, replacing $\sigma_1u_R$ with $g$ preserves the $L_4$ norm of the function. Since $u_R$ is a minimizer it
follows from \eqref{cut-offWp-norm} that
$$
||\sigma_1u_R||_{W_2^1} \le ||g||_{W_2^1} - o_R(1) = \frac{||\sigma_1u_R||_{L_4}}{||\sigma_2u_R||_{L_4}}\cdot 2^{1/4}\cdot ||\sigma_2u_R||_{W_2^1} - o_R(1).
$$
Combining \eqref{HexagonCondition} and \eqref{cut-offLq-norm} gives us
$$||\sigma_2u_R||_{L_4}^4 \le 1/4; \qquad ||\sigma_1u_R||_{L_4}^4 \ge 3/4 - o_\eps(1).$$
Hence
\begin{multline*}
\frac{||\sigma_1u_R||_{W_2^1}^2}{||\sigma_1u_R||_{L_4}^4} \le \frac{||\sigma_1u_R||_{L_4}^2}{||\sigma_2u_R||_{L_4}^2}\cdot 2^{1/2}\cdot
\frac{||\sigma_2u_R||_{W_2^1}^2}{||\sigma_1u_R||_{L_4}^4} - o_R(1) \\
\le \Big(\frac{2\cdot 1/4}{3/4-o_\eps(1)}\Big)^{1/2}\cdot \frac{||\sigma_2u_R||_{W_2^1}^2}{||\sigma_2u_R||_{L_4}^4} - o_R(1) < \frac{||\sigma_2u_R||_{W_2^1}^2}{||\sigma_2u_R||_{L_4}^4}
\end{multline*}
for $R$ large enough and $\eps$ small enough.

Finally, this shows the condition \eqref{hump_destruction_inequality} is satisfied for functions $b=\sigma_2u_R$, $c=\sigma_1u_R$ and $a=(\sigma - \sigma_1 - \sigma_2)u_R$. Proposition
\ref{hump_destruction_lemma} then gives us a function $U$ satisfying the condition \eqref{HexagonCondition} which contradicts the minimality of $u_R$.

Thus, there is at most one concentration sequence. As it was mentioned, the combined weight of concentration sequences is $1$ and the statement follows.
\end{proof}

\begin{figure}[h]
\begin{minipage}[b]{0.49\textwidth}
\centering
\includegraphics[width=6cm]{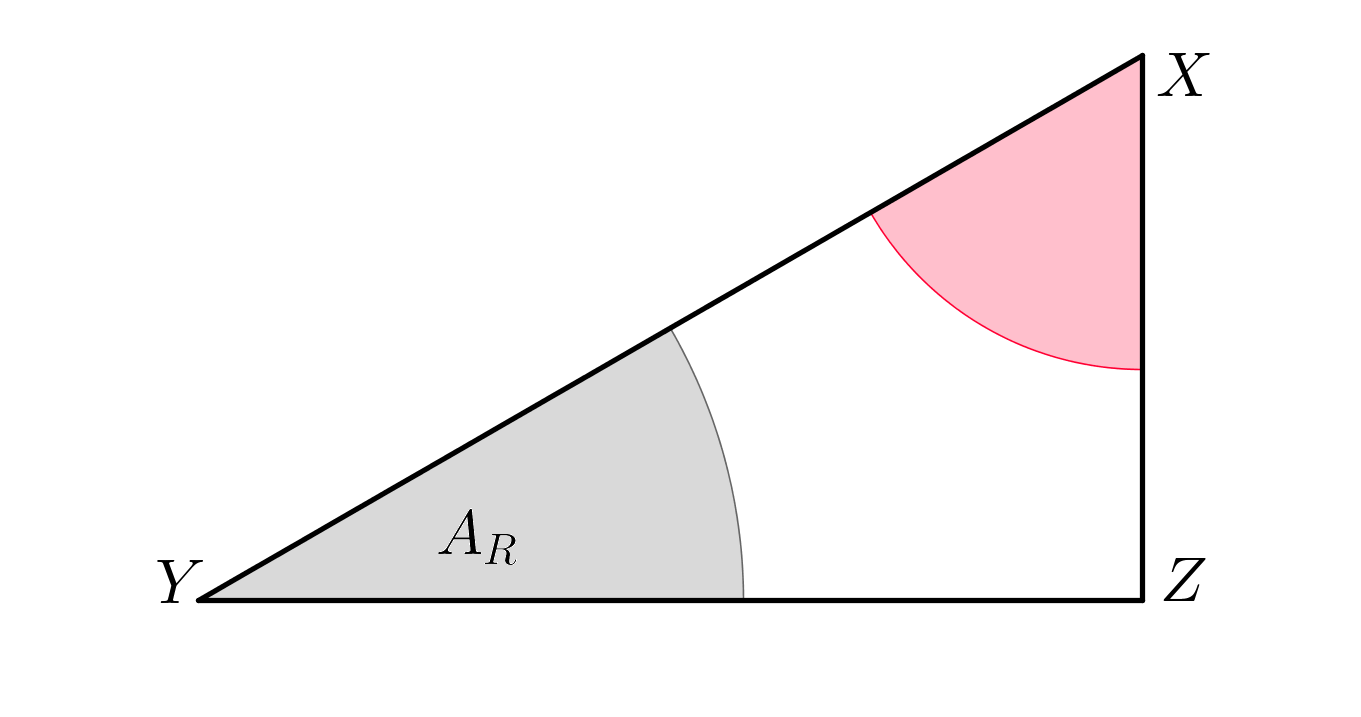}
\caption{}
\label{Fig3}
\end{minipage}
\hfill
\begin{minipage}[b]{0.49\textwidth}
\centering
\includegraphics[width=6cm]{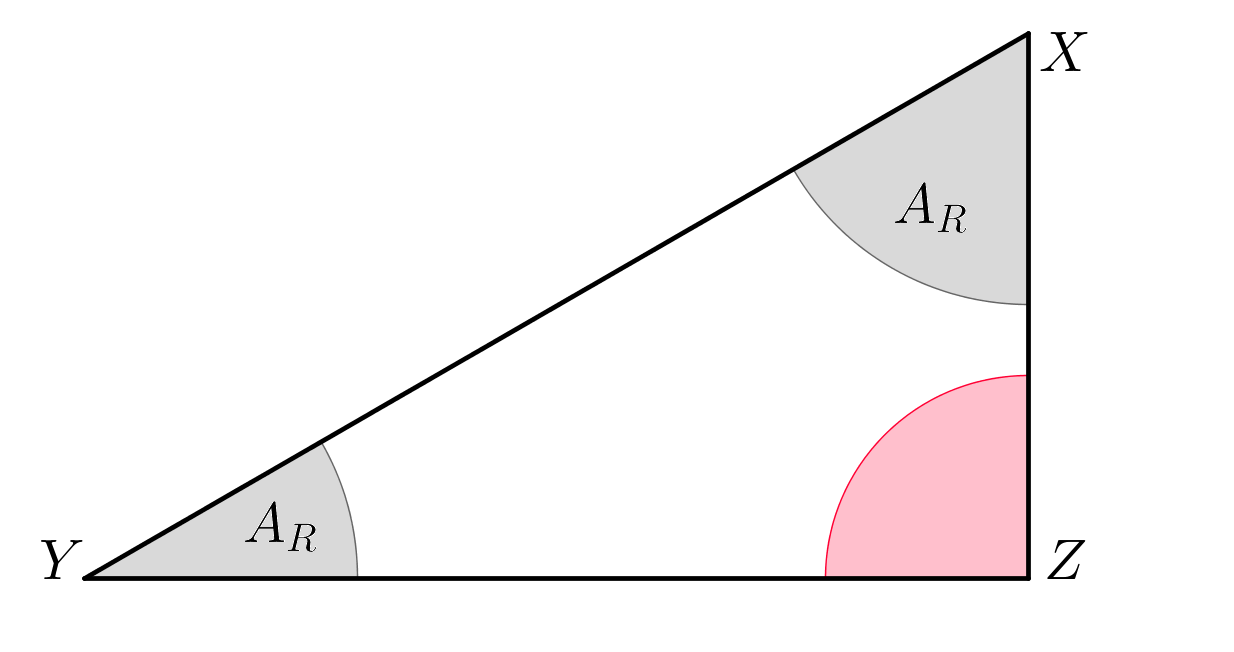}
\caption{}
\label{Fig4}
\end{minipage}
\end{figure}

Theorem \ref{conc_theorem_hexagonal} implies that
$$\int\limits_{A_R}|u_R|^4\, dx \to 0$$
as $R \to \infty$ since $u_R$ is concentrated around $X$ with weight $1$. Therefore, for $R$ large enough the restriction in \eqref{HexagonCondition} is non-active and the
Euler-Lagrange equation
\begin{equation*}
-\Delta u_R + u_R = \lambda u_R^3 \quad \mbox{in} \quad \Omega_R.
\end{equation*}
is derived the usual way.

Similar to Subsections \hyperref[ssec:General_results]{2.1} we multiply $u_R$ by $\sqrt{\lambda}$ and obtain a solution of \eqref{MainEq} concentrated near $X$. Finally, we use even
reflection to extend it to a hexagon and then extend it to $\RN^2$ periodically. The constructed solution has a hexagonal lattice of periods (see Figure~\ref{hexagonal lattice}).

\begin{figure}[h]
\begin{minipage}[b]{0.49\textwidth}
\centering
\includegraphics[width=6cm]{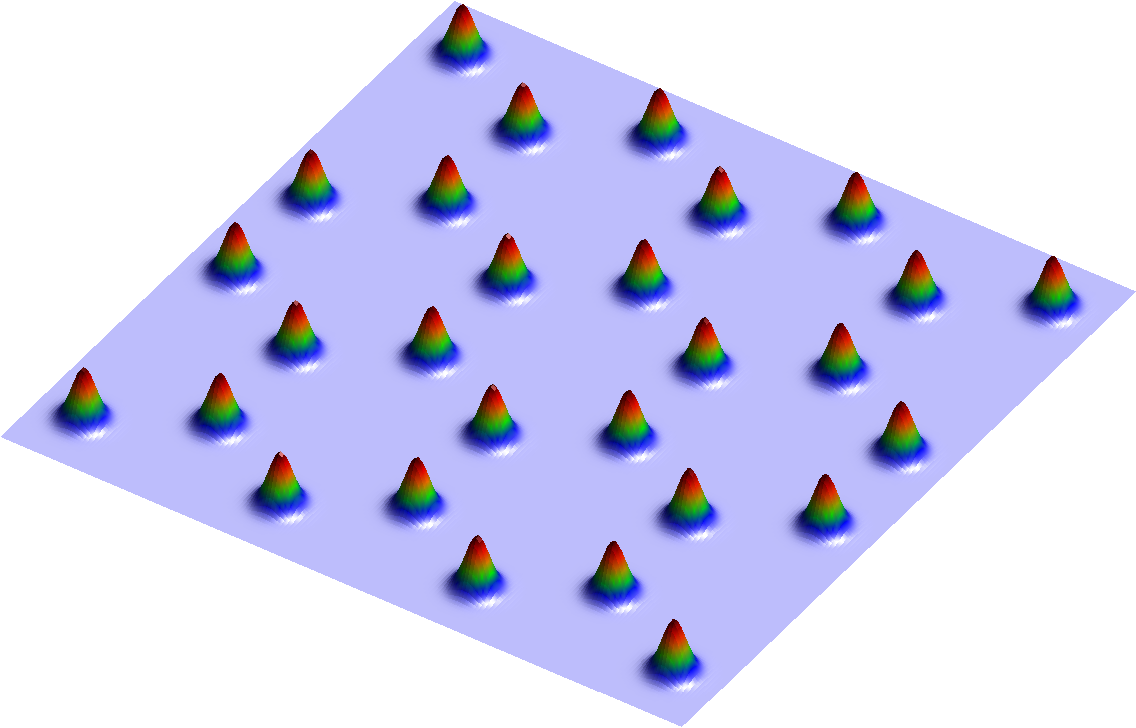}
\end{minipage}
\hfill
\begin{minipage}[b]{0.49\textwidth}
\centering
\includegraphics[width=5cm]{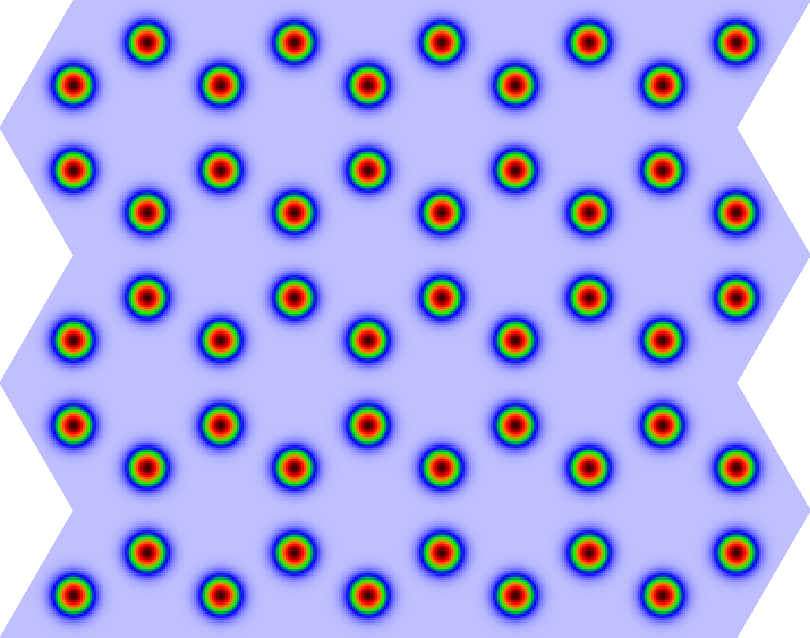}
\end{minipage}
\caption{}
\label{hexagonal lattice}
\end{figure}

\begin{rem}
Using the same technique with two constrictions (see Figure~\ref{Fig4})
\begin{equation*}
 \int\limits_{B(Y, R/4)}|u|^4\, dx \le 1/8, \; \; \int\limits_{B(X, R/4)}|u|^4\, dx \le 1/8
\end{equation*}
we can force the solution to concentrate near $Z$. This gives us yet another lattice, depicted below (see Figure~\ref{hexagonal lattice 2}).
\end{rem}

\begin{figure}[h]
\begin{minipage}[b]{0.49\textwidth}
\centering
\includegraphics[width=6cm]{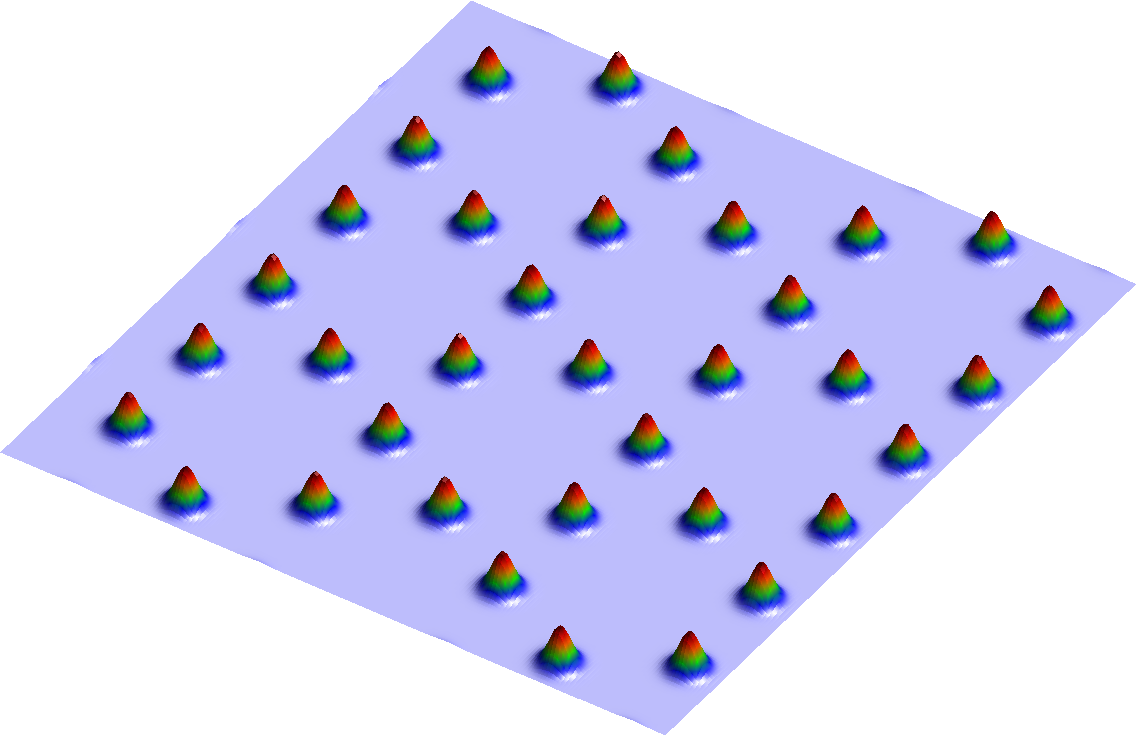}
\end{minipage}
\hfill
\begin{minipage}[b]{0.49\textwidth}
\centering
\includegraphics[width=5cm]{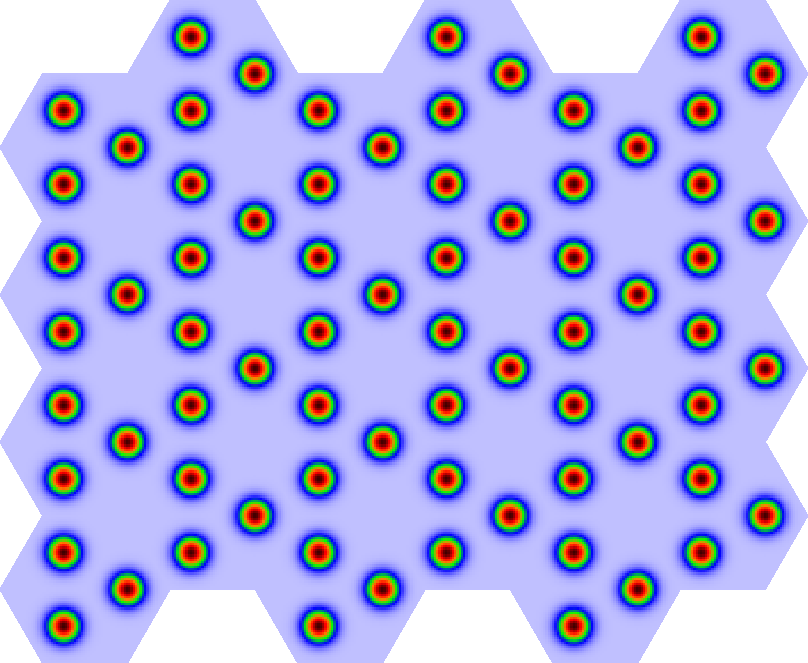}
\end{minipage}
\caption{}
\label{hexagonal lattice 2}
\end{figure}

\subsection{Breather type solutions}
\label{ssec:Breather}
Here we construct a family of solutions periodic in one variable and rapidly decaying in another. To do this we consider problem \eqref{ConMinProblem} in the strip
$\Omega_R=(0,R)\times\mathbb{R}$. Since the embedding of $W_2^1(\Omega_R)$ into $L_4(\Omega_R)$ is not compact we cannot apply the argument from Subsection
\hyperref[ssec:General_results]{2.1} directly and need to refine it.

Let $v_n$ be a minimization sequence for the problem \eqref{ConMinProblem}:
$$||v_n||_{L^4} = 1; \qquad ||v_n||_{W_2^1} \searrow \min \quad \mbox{as} \ n \to \infty$$
Extending the functions to $\RN^2$ and using Lemma \ref{no_vanishing_lemma} shows there is no vanishing. Therefore, $v_n$ must satisfy the concentration condition.

Notice that the Steiner symmetrization with respect to $y$ and monotonous rearrangement with respect to $x$ do not increase the functional \eqref{VarProblem}
(see, e.g., \cite[II.7]{Ka}). Taking into account that $J[|v|] = J[v]$, we can presume that $v_n$ are nonnegative, symmetrically decreasing in $y$ and monotonous
(without loss of generality, decreasing) in $x$. Then they can only concentrate around $(0, 0)$.

Next, we extract a subsequence $v_{n_k}$ which converges weakly in $W_2^1(\Omega_R)$ to some function $u_R$. Let $T > 0$ and consider $\Omega_{R,T} = (0, R) \times (-T, T)$.
The sequence $v_{n_k}$ (restricted to $\Omega_{R, T}$) converges weakly in $W^1_2(\Omega_{R, T})$. Since the embedding of $W^1_2(\Omega_{R, T})$ into $L_4(\Omega_{R, T})$ is
compact it converges strongly in $L_4(\Omega_{R, T})$. Choose $\eps > 0$. The sequence $v_{n_k}$ is concentrated around $(0, 0)$ which shows there is $T > 0$ such that
$$
||v_{n_k}||_{L_4(\Omega_{R})} \ge ||v_{n_k}||_{L_4(\Omega_{R, T})} > (1-\eps)||v_{n_k}||_{L_4(\Omega_{R})}.
$$
Therefore,
$$
\liminf ||v_{n_k}||_{L_4(\Omega_{R})} \ge ||u_R||_{L_4(\Omega_{R, T})} \ge \limsup (1-\eps)||v_{n_k}||_{L_4(\Omega_{R})}.
$$
\begin{figure}
\centering
\includegraphics[width=0.45\textwidth]{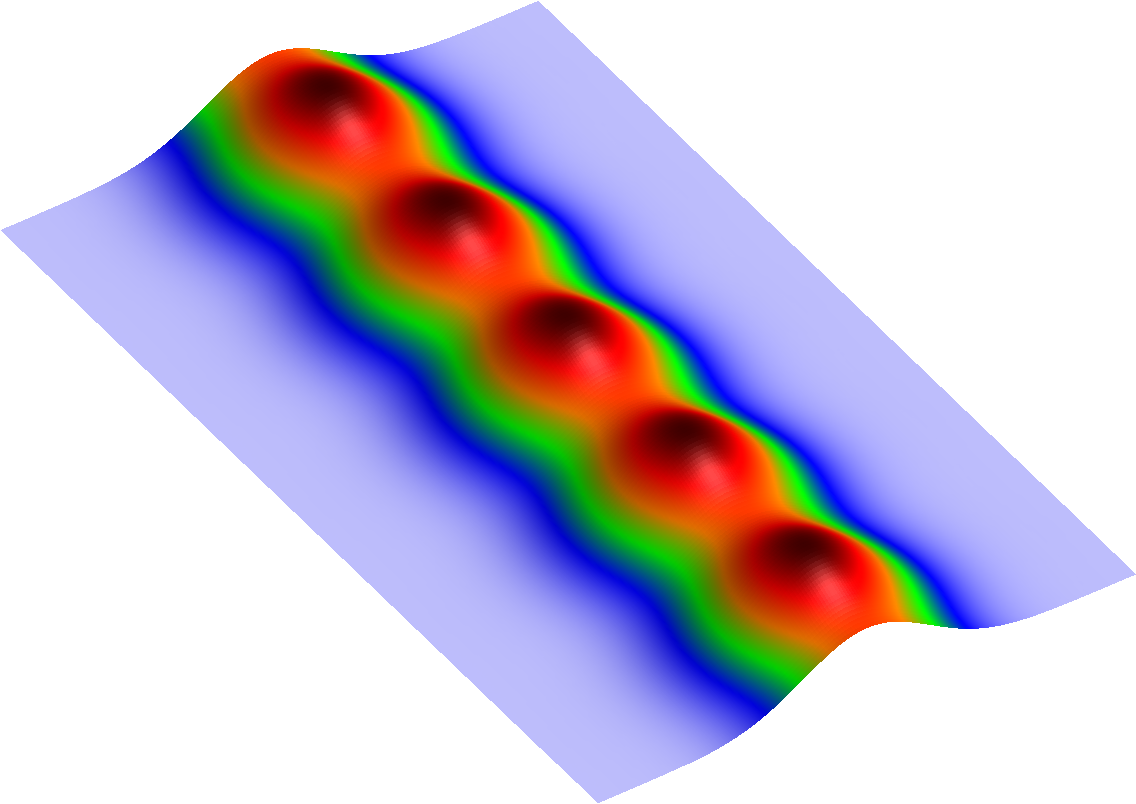}
\includegraphics[width=0.45\textwidth]{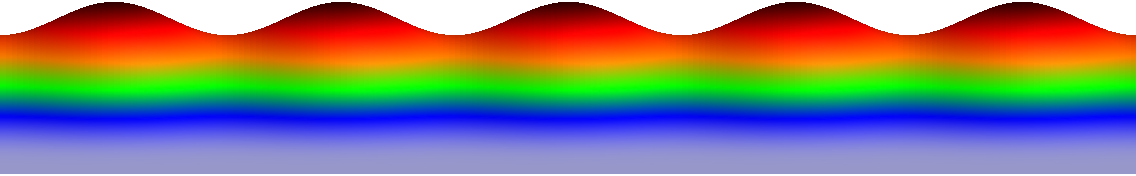}
\caption{}
\label{breather type}
\end{figure}
Thus $||u_R||_{L_4(\Omega_R)} = \lim_{T \to \infty} ||u_R||_{L_4(\Omega_{R, T})} = 1$. Finally, since $u_R$ is the weak limit of $v_{n_k}$, $||u_R||_{W_2^1} \le \liminf ||v_{n_k}||_{W_2^1}$
and $u_R$ is, therefore, a minimizer. By construction, $u_R$ is non-constant in $x$ provided $R$ is large enough.

Now we extend $u_R$ to the whole plane by even reflection and periodic expansion. This gives us a required solution of (\ref{MainEq}) in $\RN^2$ (see Figure~\ref{breather type}).

As a corollary, given $N\in\mathbb{N}$, for $R>R^*(N)$ we obtain at least $N$ different nontrivial positive solutions of
(\ref{MainEq}) in $\mathbb{R}^2$, $2R$-periodic in $x$ and symmetrically decreasing in $y$.

\medskip

\subsection{Sign-changing solutions}
Now we consider problem \eqref{VarProblem} in the rectangle $(0, R) \times (0, aR)$ with the additional condition of $u=0$ on $\{0, R\} \times (0, aR)$ (the vertical sides of the rectangle). Similarly to the previous sections we deduce the solution is concentrated in a half-circle adherent to the horizontal side. We again extend it to the strip $(0, R) \times \RN$ by even reflection and then to the strip $(0, 2R) \times \RN$ by odd reflection. Next we extend the function to $\RN^2$ periodically which gives us a solution with alternating signs (see Figure~\ref{rectangular lattice sign 1}).

We can also consider the boundary condition $u=0$ on all the boundary of the rectangle which will give us a solution concentrated in a circle. We extend it to $(-R, R) \times (-aR, aR)$
oddly and to $\RN^2$ periodically. The resulting signs are in staggered order (see Figure~\ref{rectangular lattice sign 2}). Breather type case is analogous 
(see Figure~\ref{breather type sign 1}).

\begin{figure}[h]
\begin{minipage}[b]{0.49\textwidth}
\centering
\includegraphics[width=6cm]{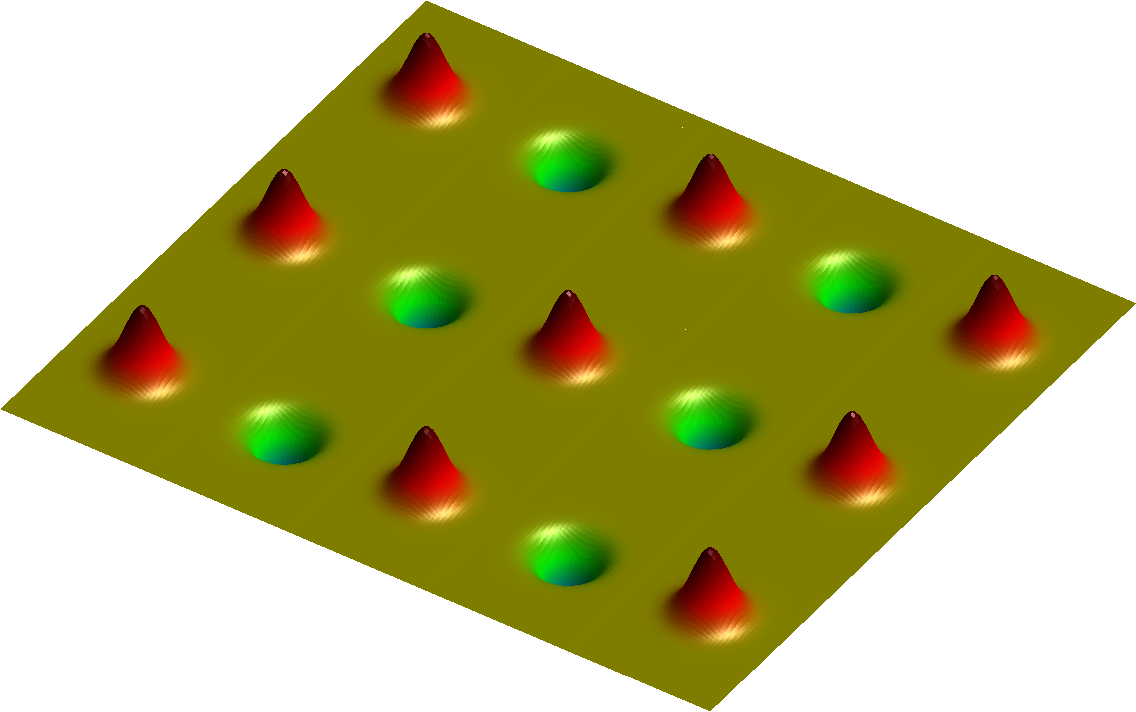}
\end{minipage}
\hfill
\begin{minipage}[b]{0.49\textwidth}
\centering
\includegraphics[width=7cm]{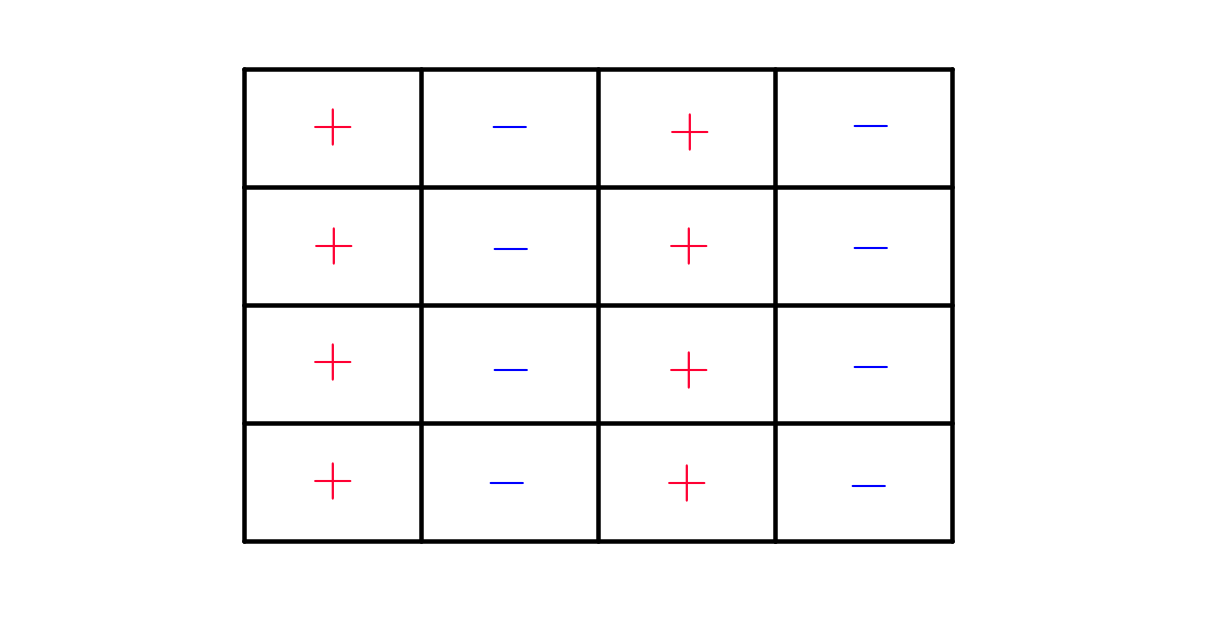}
\end{minipage}
\caption{}
\label{rectangular lattice sign 1}
\end{figure}

\begin{figure}[h]
\begin{minipage}[b]{0.49\textwidth}
\centering
\includegraphics[width=6cm]{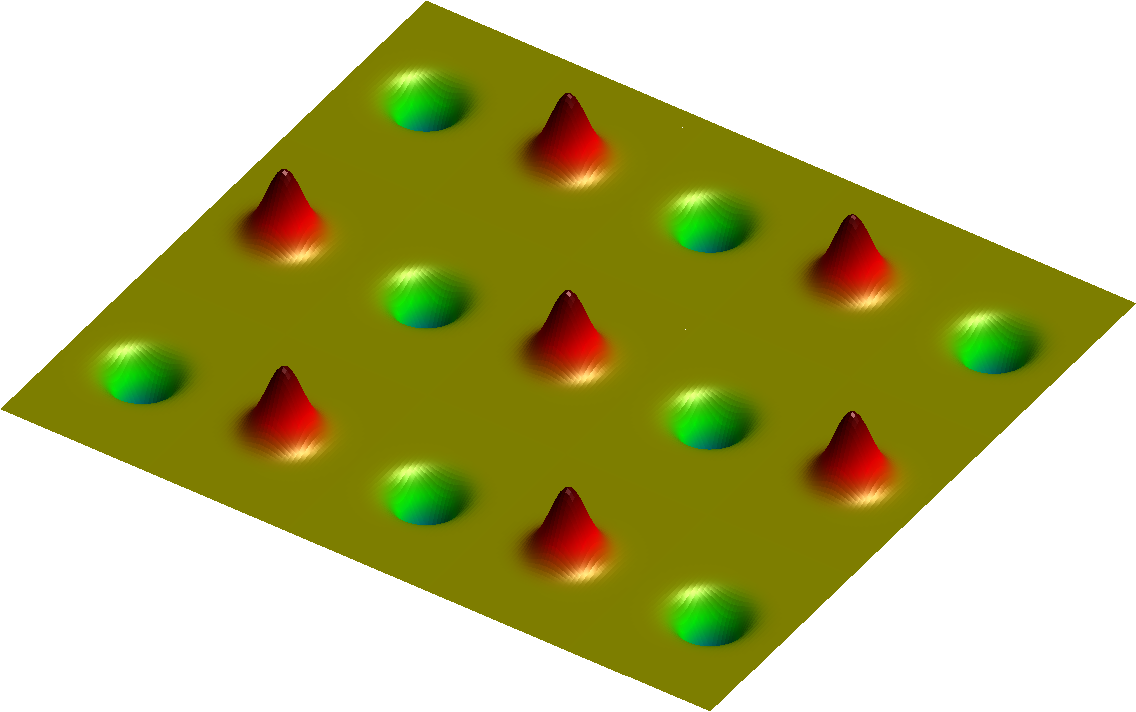}
\end{minipage}
\hfill
\begin{minipage}[b]{0.49\textwidth}
\centering
\includegraphics[width=7cm]{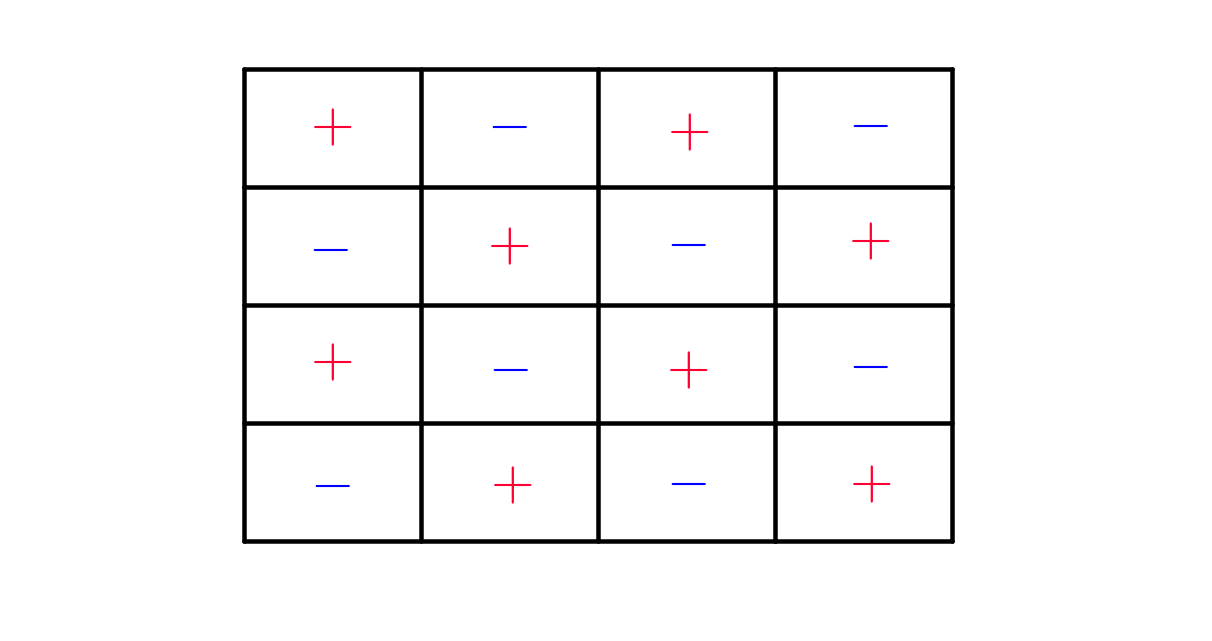}
\end{minipage}
\caption{}
\label{rectangular lattice sign 2}
\end{figure}

\begin{figure}[h]
\begin{minipage}[b]{0.49\textwidth}
\centering
\includegraphics[width=6cm]{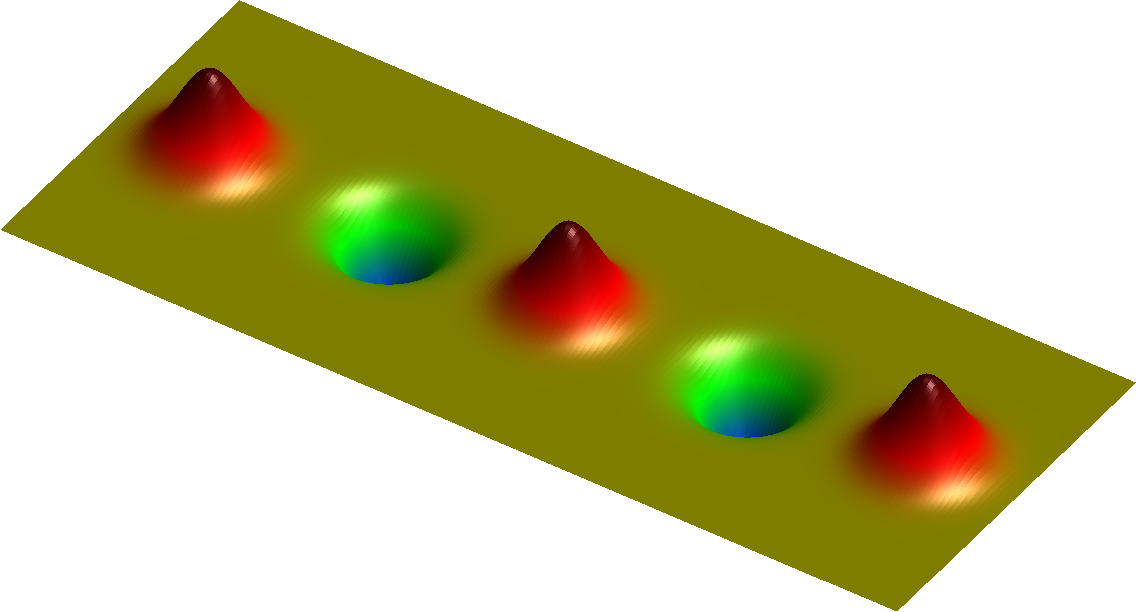}
\end{minipage}
\hfill
\begin{minipage}[b]{0.49\textwidth}
\centering
\includegraphics[width=6cm]{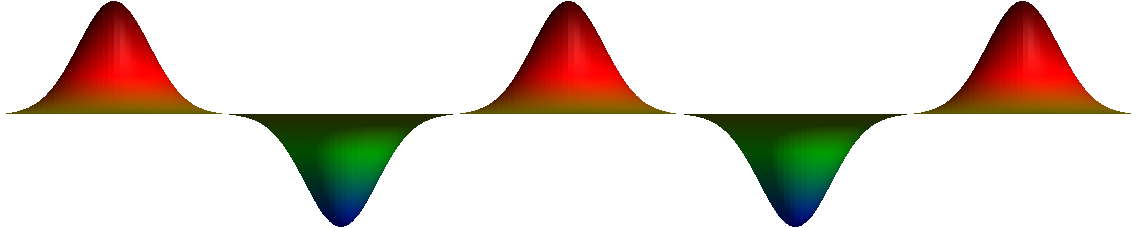}
\end{minipage}
\caption{}
\label{breather type sign 1}
\end{figure}

The equilateral triangular case is a little different. If the condition $u=0$ holds for only one side of the triangle then the minimizer is concentrated in the opposite corner.
We extend it to the hexagon and the function is positive there. We could try to extend it to the whole plane but the partitioning of $\RN^2$ into hexagons have three hexagons
which are pairwise adherent and each pair must have different signs which is impossible.

However if we set $u=0$ on the whole boundary of the triangle we do not have this problem. We use odd reflection and extend the function to $\RN^2$ like on the picture
(see Figure~\ref{hexagonal lattice sign 1}). Doing the same with triangles that have angles $\pi/2, \pi/4, \pi/4$ and $\pi/2, \pi/3, \pi/6$ gives us two more types of
solutions (see Figures~\ref{rectangular lattice sign 3}, \ref{hexagonal lattice sign 2}).

\begin{figure}[h]
\begin{minipage}[b]{0.49\textwidth}
\centering
\includegraphics[width=7cm]{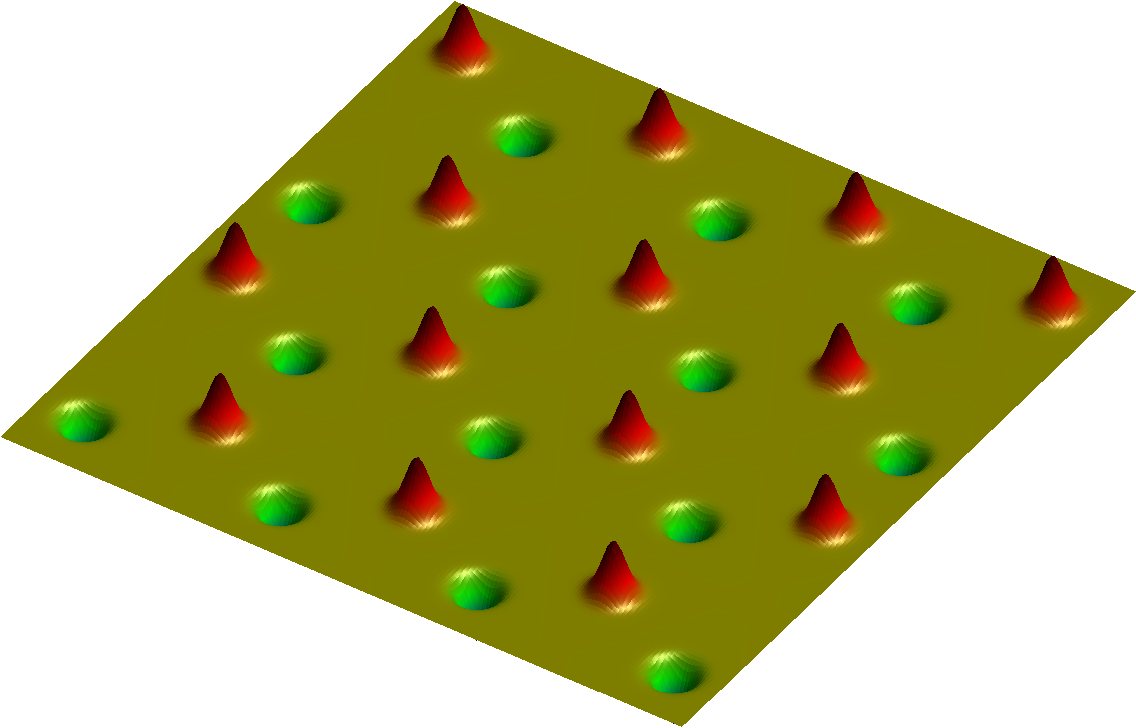}
\end{minipage}
\hfill
\begin{minipage}[b]{0.49\textwidth}
\centering
\includegraphics[width=9cm]{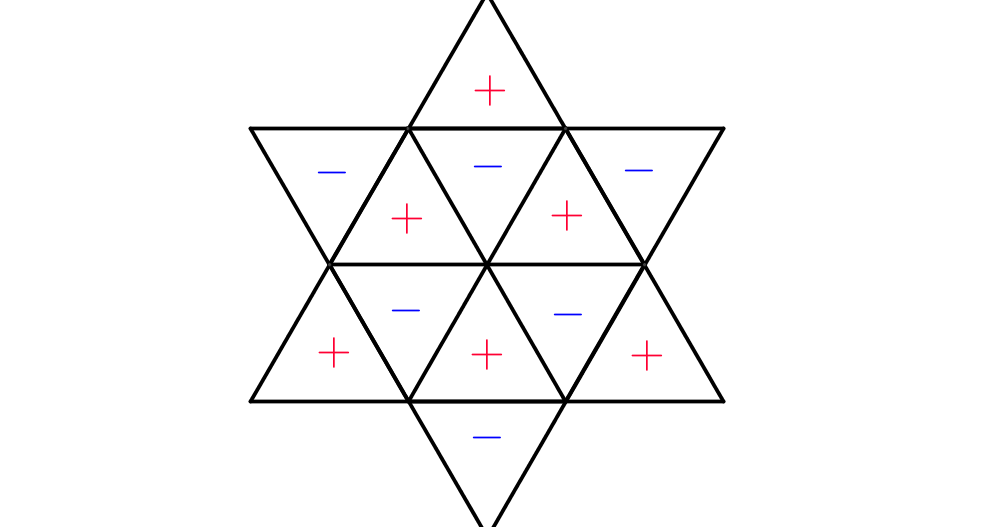}
\end{minipage}
\caption{}
\label{hexagonal lattice sign 1}
\end{figure}

\begin{figure}[h]
\begin{minipage}[b]{0.49\textwidth}
\centering
\includegraphics[width=7cm]{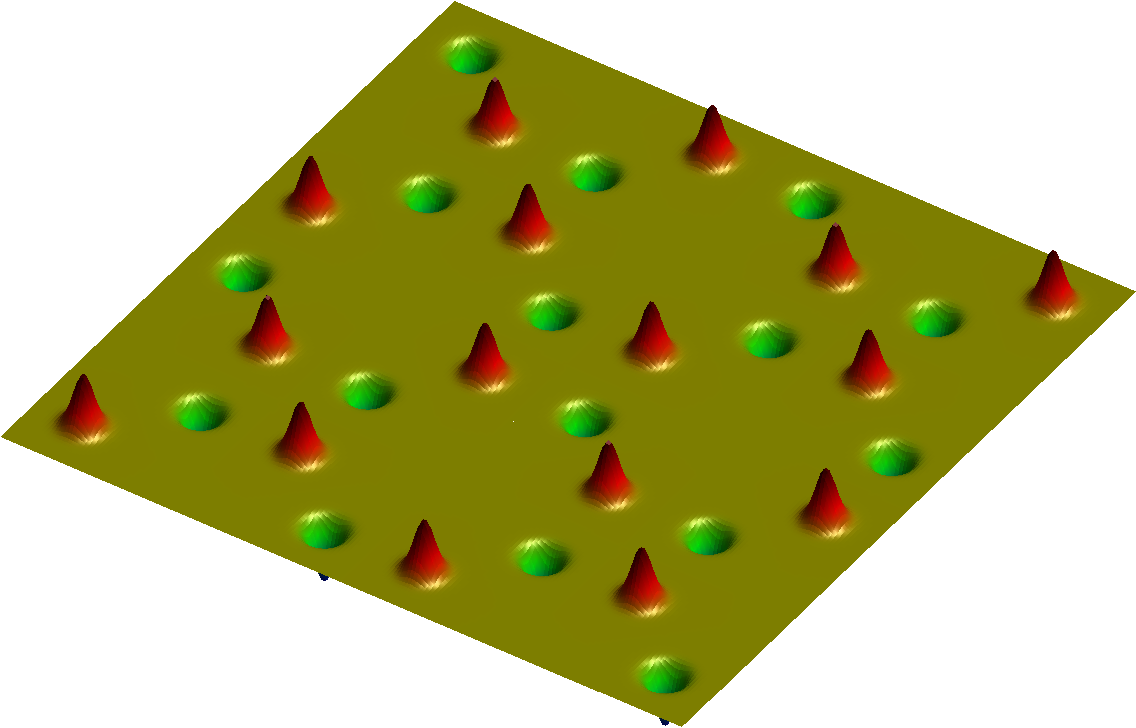}
\end{minipage}
\hfill
\begin{minipage}[b]{0.49\textwidth}
\centering
\includegraphics[width=8cm]{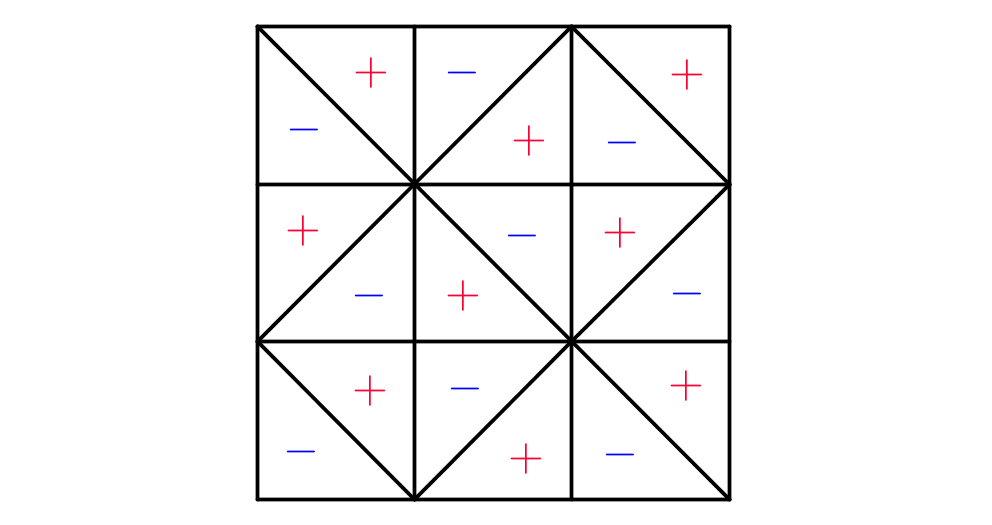}
\end{minipage}
\caption{}
\label{rectangular lattice sign 3}
\end{figure}

\begin{figure}[h]
\begin{minipage}[b]{0.49\textwidth}
\centering
\includegraphics[width=7cm]{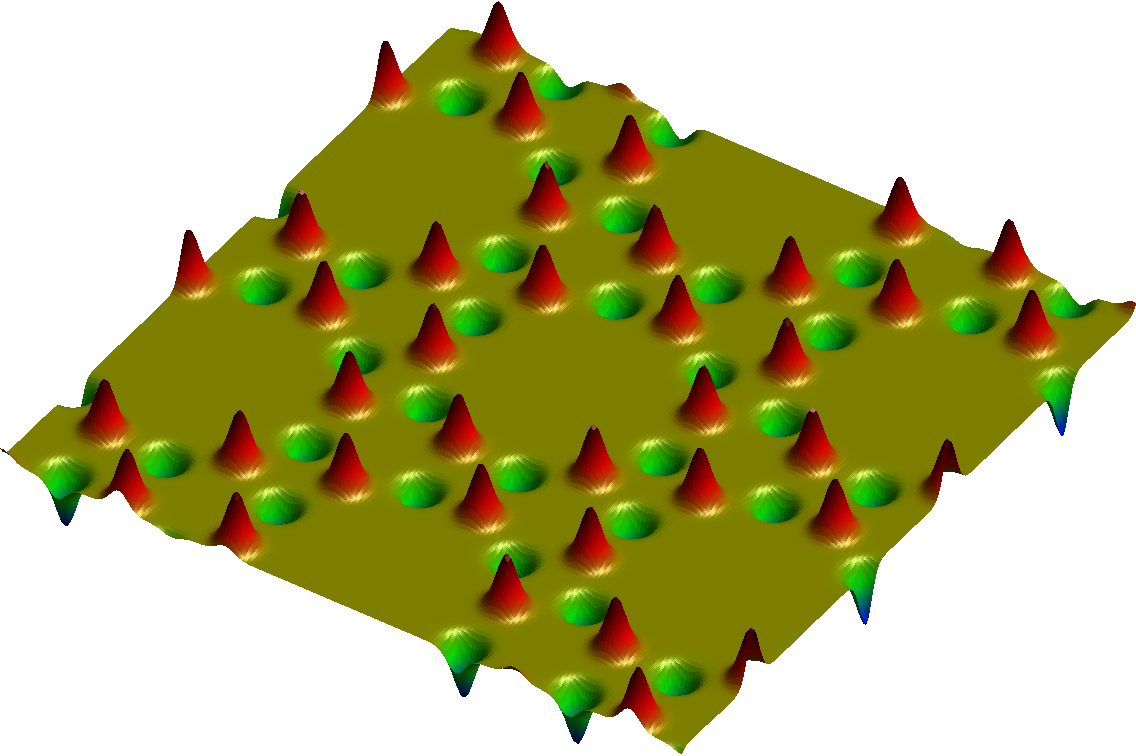}
\end{minipage}
\hfill
\begin{minipage}[b]{0.49\textwidth}
\centering
\includegraphics[width=9cm]{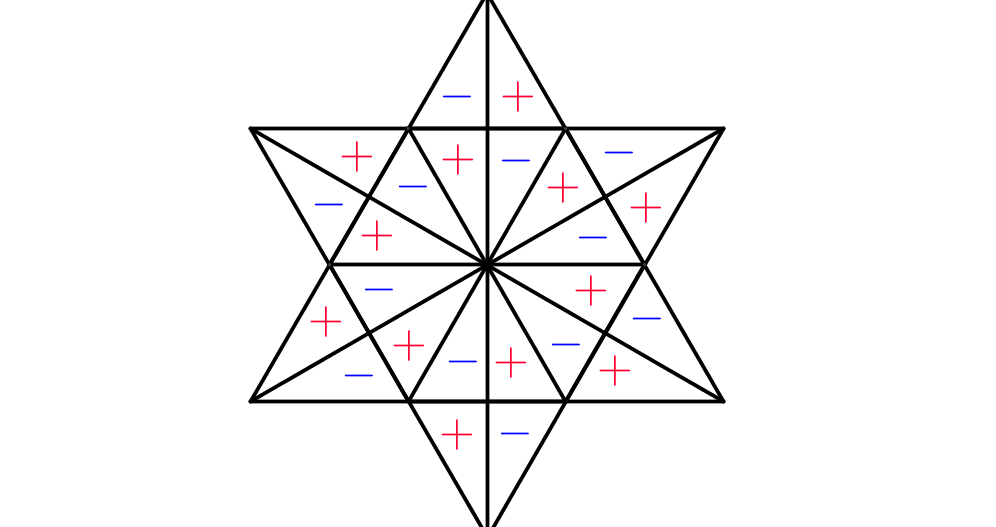}
\end{minipage}
\caption{}
\label{hexagonal lattice sign 2}
\end{figure}

\section{Radially symmetric solutions}
In this section we prove the existence of a positive radially symmetric solution using the standard method and the existence of a countable family of radially symmetric solutions using
Lusternik-Schnirelmann Theorem. Notice that similar results were proved by the methods of the theory of dynamical systems in \cite{Jones}.	

\subsection{Positive solution}

Here we consider the problem \eqref{ConMinProblem} in $\Omega = \Omega_R = \RN^2$. As in the breather type case we can find a minimization sequence $v_n$ which is radially symmetric
since symmetrical rearrangement does not increase the functional \eqref{VarProblem} (see, e.g.,\cite[II.9]{Ka}). We again note that $v_n$ must concentrate around $(0, 0)$. By considering
the problem in balls $B(0, T)$ and taking $T$ to infinity we prove the nonzero limit exists similarly to Subsection \hyperref[ssec:Breather]{2.5}. In the end this gives us a positive,
symmetrically decreasing solution of the problem \eqref{MainEq} in $\RN^2$.

\subsection{A countable family}
To prove that problem \eqref{MainEq} has a countable number of radial solutions we need the
following statement (the Lusternik-Schnirelmann theorem, see, e.g., \cite[Chapter 8]{osm}).

\begin{propos}
\label{Lusternik_Schnirelmann}
Suppose $H$ is a Hilbert space and $I : H \to \RN$ is a (nonlinear) functional such that:
\begin{enumerate}
\item $I$ and is weakly continuous and smooth (namely, $I\in C^{1, 1}_{loc}$),
\item $I$ is even and $I[0]=0$,
\item $I[u]>0$ and $||I^{\prime}[u]||>0$ for $u \ne 0$.
\end{enumerate}
Then $I$ has at least a countable number of critical points on the sphere $S_a =
\{x \in H \mid ||x||=a\}$ for every $a>0$.
\end{propos}

We take the subspace of radial functions in $W_2^1(\RN^2)$ as $H$ and put $I[u] =
\int_{\RN^2}u^4$. We now show that $I$ satisfies the conditions of Proposition
\ref{Lusternik_Schnirelmann}. The conditions 2 and 3 are evident. The map
$u \mapsto I^{\prime}[u]$ is Lipschitz on any bounded set in $H$ so the functional is
$C^{1, 1}_{loc}$.
It remains to prove that $I$ is weakly continuous. We pass to the polar coordinates and
write
$$||u||^2_H = 2\pi \int_0^\infty r(|u|^2 + |u^\prime|^2)\, dr.$$
$$||u||^4_{L_4} = 2\pi \int_0^\infty r |u|^4 \, dr.$$
Using an obvious inequality
$$||u||_{L_4(R, R+1)} \le C ||u||_{W_2^1(R, R+1)}$$
we obtain for $R \ge 1$
\begin{multline*}
\Big(\int_R^{R+1} r |u|^4 \, dr\Big)^{1/2} \le (R+1)^{1/2}CR^{-1} \int_R^{R+1} r(|u|^2 +
|u^\prime|^2) \, dr \\
= C\Big(\frac{R+1}{R^2}\Big)^{1/2}\int_R^{R+1} r(|u|^2 + |u^\prime|^2) \, dr.
\end{multline*}
This implies, similar to Lemma \ref{no_vanishing_lemma} (Appendix \hyperref[ssec:App A]{A}),
\begin{multline}
\label{NoGoodName}
\int_R^{\infty}r |u|^4 \, dr = \sum_{k=0}^{\infty} \int_{R+k}^{R+k+1}r |u|^4 \, dr \\
\le \Big(\sup_k \int_{R+k}^{R+k+1}r |u|^4 \, dr\Big)^{1/2}\sum_{k=0}^{\infty}
\Big(\int_{R+k}^{R+k+1}r |u|^4 \, dr\Big)^{1/2} \\
\le C\Big(\sup_k \int_{R+k}^{R+k+1}r |u|^4 \, dr\Big)^{1/2}\sum_{k=0}^{\infty}
\bigg(\frac{R+k+1}{(R+k)^2}\bigg)^{1/2}\int_{R+k}^{R+k+1} r(|u|^2 + |u^\prime|^2) \, dr \\
\le C \frac{R+1}{R^2} \Big(\int_R^\infty r(|u|^2 + |u^\prime|^2) \, dr\Big)^2.
\end{multline}
It is well known that the embeddings $W_2^1(B(0, R)) \emb L_4(B(0, R))$ are compact. It follows
that the mappings $u \mapsto \left. u \right|_{B_R(0)}$ from $H$ to $L_4(\RN^2)$ are compact.
By \eqref{NoGoodName} the embedding $H\emb L_4(\RN^2)$ is the norm limit of compact operators
hence compact which proves the weak continuity of $I$.

The critical points of $I$ on the sphere $S_a$ are those where
\begin{equation}
\label{radial}
\int\limits_{\RN^2} u^3h = \lambda\int\limits_{\RN^2}(\nabla u\nabla h + uh)
\end{equation}
for some $\lambda \in \RN$ and every $h \in H$.

By the principle of symmetric criticality, see \cite{Pal}, the relation (\ref{radial}) holds
for any $h\in W^1_2(\RN^2)$.
Taking $h=u$ shows that $\lambda > 0$. Therefore we can multiply $u$ by $\sqrt{\lambda}$ and get
a solution of \eqref{MainEq} in $\RN^2$.
Thus, Lusternik-Schnirelmann theorem implies that there is a countable number of radial solutions.

\section{Some generalizations}
Our arguments are valid if we consider the equation \eqref{MainEq} in $\RN^3$. By choosing an
appropriate domain $\Omega$ we can get the following:
\begin{enumerate}
\item solutions periodic in $x, y, z$;

\item solutions triangular-periodic in $x, y$ and periodic in $z$;

\item solutions periodic in $x, y$ and symmetrically decreasing in $z$;

\item solutions triangular-periodic in $x, y$ and symmetrically decreasing in $z$;

\item solutions periodic in $x$ and symmetrically decreasing in $y, z$;

\item radial solutions etc.
\end{enumerate}

More generally, consider the equation
\begin{equation}
\Delta_pu - |u|^{p-2}u + |u|^{q-2}u = 0 \quad \mbox{in} \quad \RN^n
\label{GenerEq}
\end{equation}
Here $1 < p < \infty$, $\Delta_pu \equiv \mathrm{div}(|\nabla u|^{p-2}u)$ is a $p$-Laplacian
while $q \in (p, p^*)$.
The corresponding variational problem is the minimization of the functional $J[u] =
||u||_{W_p^1}/||u||_{L_q}$.
Since the Concentration Theorem (Theorem \ref{Concentration_thr} in Appendix
\hyperref[ssec:App A]{A}) holds true, the argument similar to the one in Section 2 can be applied
again.
In that way we obtain positive solutions of \eqref{GenerEq} which have various periodic lattices
in some variables and are symmetrically decaying in other variables.
The sign-changing solutions with various periodic structures could be obtained as well.

Using the classical Nehari method, it is possible to apply our machinery for a more general
equation
\begin{equation}
\Delta_pu - |u|^{p-2}u + f(u) = 0 \quad \mbox{in} \quad \RN^n
\label{NehariMethodEquation}
\end{equation}
with an odd function $f$ satisfying some natural assumptions. Roughly speaking, $f(s)$ is assumed
to be ``more convex'' than $s^p$ for $s > 0$ and to have subcritical growth at infinity.

For instance, the requirements for $f$ can be given as follows:
\begin{gather*}
sf^{\prime}(s) > (p-1)f(s) \ \mbox{for almost any} \ s \ge 0; \\\
\liminf\limits_{s \to \infty}\frac{sf(s)}{\int_0^s f(t) \, dt} > p;\\
\lim\limits_{s \to 0} \frac{f(s)}{s^{p-1}} = 0; \\
\lim\limits_{s \to \infty}\frac{sf(s)}{\Phi(s)} = 0,
\end{gather*}
where
\begin{equation*}
\begin{aligned}
\Phi(s) &= s^{p^*},  && p < n, \\
\Phi(s) &= s^q \ \mbox{for any} \ q \in (p, \infty), && p\ge n.
\end{aligned}
\end{equation*}

The method used to prove the existence of solutions is well known (see e.~g. \cite{YYLi}, \cite{N04}).
After the solution is found the concentration theory can be applied (with appropriate modifications) and the results follow.

\section{Appendix A}
\label{ssec:App A}

\subsection{Concentration}
\begin{propos}[a variant of Lemma 1.1 from \cite{Lio}]
 \label{point_concentration_lemma}
Suppose that $G(s)$, $s\in\RN$ is a positive function. Consider a sequence of functions $u_j(x)$, $x\in\mathbb{R}^n$ and suppose $\int_{{\RN^n}} G(u_j) dx$ is finite for every $j$.
Then (up to a subsequence) one of the two conditions is satisfied:

\begin{enumerate}
\item ({\it concentration}) There exist $\lambda\in (0,1]$ and a sequence of points $x_j \in \RN^n$
such that for every $\varepsilon > 0$ there exist $\rho>0$, a sequence $\rho'(j) \to \infty$ and a number $j_0$ such that for every $j \geqslant j_0$
\begin{multline}
 \label{concentration}
 \bigg | \int\limits_{B(x_j,\rho)} G(u_j)\, dx - \lambda \int\limits_{{\mathbb{R}^n}} G(u_j)\, dx \bigg | \\
 +\ \bigg | \int\limits_{\; \mathbb{R}^n \setminus B(x_j,\rho'(j))} \! G(u_j)\, dx - (1 - \lambda) \int\limits_{{\mathbb{R}^n}} G(u_j)\, dx \bigg |
 < \varepsilon \int\limits_{{\mathbb{R}^n}} G(u_j)\, dx.
\end{multline}
In that case $x_j$ is called a concentration sequence of $G(u_j)$ and $\lambda$ is called a weight of the sequence.

\item ({\it vanishing}) For every $\rho>0$ the following equality holds:
\begin{equation}
 \label{vanishing}
 \lim_{j\to\infty} \sup\limits_{x\in \RN^n}\int\limits_{B(x,\rho)} G(u_j)\, dx =0.
\end{equation}
\end{enumerate}
\end{propos}
\begin{rem}
\label{RhoDecreaseRem}
The condition (\ref{concentration}) remains true if we substitute the sequence of radii $\rho'(j)$ with a smaller sequence which tends to infinity and $\rho$ with a larger constant.
\end{rem}

\begin{rem}
\label{BoundShiftRem}
If $x_j$ is a concentration sequence and $y_j$ is a sequence of points such that $|x_j-y_j| \le d$
for some positive $d \in \RN$, then $y_j$ is a concentration sequence as well.
Indeed, it is easy to see that $\rho_y = \rho_x + d$ and $\rho_y\sht(j) = \rho_x\sht(j)-d$ makes
it satisfy \eqref{concentration}. In this case sequence $y_j$ is called
equivalent to $x_j$.
\end{rem}

\subsection{Some lemmata}
\label{ssec:general_theory}
\begin{propos}[A variant of lemma 1.6 from \cite{kolprepr}]
\label{hump_destruction_lemma}
Assume that $1<p<q<\infty$ and functions $a, b, c \in W^1_p(\Omega) \cap L_q(\Omega)$ have separated supports. Suppose also that $b \not\equiv 0$, $c \not\equiv 0$ and
\begin{equation}
\label{hump_destruction_inequality}
\frac{||b||_{W_p^1}^p}{||b||_{L_q}^q} \ge \frac{||c||_{W_p^1}^p}{||c||_{L_q}^q}.
\end{equation}
Let
$$
u = a + b + c
$$
and
$$
U = a + \frac{(||b||_{L_q}^q + ||c||_{L_q}^q)^{1/q}}{||c||_{L_q}}\cdot c.
$$
Then one has
\begin{gather*}
 U=a \mbox{ for } x \in \Omega\setminus({\rm supp \ } c);\\
 \int \limits_{\Omega} U^q\, dx = \int \limits_{\Omega} u^q\, dx;\\
 ||U||_{W_p^1}^p < ||u||_{W_p^1}^p - C(b,c). \\
\end{gather*}
Furthermore, the constant $C(b,c)$ depends only on $\|b\|_{L_q(\Omega)}, \|c\|_{L_q(\Omega)},
\|b\|_{W^1_p(\Omega)}$ and $\|c\|_{W^1_p(\Omega)}$.
\end{propos}
\begin{proof}
The proof relies on a simple calculation.
\end{proof}

The following lemma was proved in \cite{ByTan} (Lemma 2.9) for the case $p=2$. The general case
is very similar but for convenience's sake we prove it here.
\begin{lemma}
\label{no_vanishing_lemma}
Suppose a sequence $u_R$ is bounded in $W^1_p (\RN^n)$,  $1 < p < q < p^*$ and for some $\rho > 0$
$$
\lim\limits_{R \to \infty} \sup\limits_{x \in \omega}\int\limits_{B(x, \rho)}|u_R|^q\, dx = 0
$$
where $\omega$ is an open subset in $\RN^n$.

Then one has
$$
\int\limits_{\omega}|u_R|^q\, dx \to 0 \quad \mbox{as} \quad R \to \infty.
$$

In case $p < n$ the statement is also true for $q = p^*$.
\end{lemma}

\begin{proof}

Let $d$ be a positive real number less than $\rho/\sqrt{n}$.

For $m = (m_1, \ldots, m_n) \in \ZN^n$ we write
$$
Q_m = [m_1d , (m_1+1)d] \times [m_2d , (m_2+1)d] \times \cdots \times [m_nd , (m_n+1)d].
$$

By the Sobolev inequality, there exists $C>0$ such that
$$||u||_{L_q(Q_m)} \le C ||u||_{W_p^1(Q_m)}$$
for all $u \in W_p^1(Q_m)$.

Let $M$ be the set of $m \in \ZN^n$ such that $Q_m \cap \omega \ne \emptyset$ and let $\Omega = \bigcup\limits_{m \in M}Q_m$.

For any $u \in W_p^1(\RN^n)$ we deduce
\begin{multline}
||u||^q_{L_q(\Omega)}  = \sum\limits_{m \in M} ||u||^q_{L_q(Q_m)} \le (\sup\limits_{m \in M}||u||_{L_q(Q_m)})^{q-p}\sum\limits_{m \in M} ||u||^p_{L_q(Q_m)} \\
\le C^p (\sup\limits_{m \in M}||u||_{L_q(Q_m)})^{q-p} \sum\limits_{m \in M} ||u||^p_{W_p^1(Q_m)} = C^p (\sup\limits_{m \in M}||u||_{L_q(Q_m)})^{q-p} ||u||^p_{W_p^1(\Omega)} \\
\le C^p (\sup\limits_{m \in M}||u||_{L_q(Q_m)})^{q-p} ||u||^p_{W_p^1(\RN^n)}.
\label{Lem1ineq}
\end{multline}

By the choice of $d$, $Q_m \subset B(x, \rho)$ for all $x \in Q_m$. Therefore, for $m \in M$ and $x \in Q_m \cap \omega$, we have
$$
||u||_{L_q(Q_m)} \le ||u||_{L_q(B(x, \rho))},
$$
and $\sup_{m \in M}||u_R||_{L_q(Q_m)} \to 0$ by hypothesis.

Since $u_R$ is bounded in $W_p^1(\RN^n)$ it follows from \eqref{Lem1ineq} that $||u_R||_{L_q(\Omega)} \to 0.$
Since $\omega \subset \Omega$ we are done.
\end{proof}

\begin{lemma}[an analogue of Lemma 3.1 from \cite{kolprepr}]
\label{cut_off_lemma}
Suppose a sequence of functions $\{u_R\}$ is normalized in $L_q(\Omega_R)$ and bounded in $W_p^1(\Omega_R)$.
Let $\{x_R\}$ be a concentration sequence of $|u_R|^q$, i.e. for every $\eps > 0$ there exists a radius $\rho > 0$
and a sequence of radii $\rho\sht (R)$ that satisfy concentration condition for $G(s) = |s|^q$. Define $\sigma$ as a cut-off function that satisfies
\begin{gather*}
\sigma (x) =
\begin{cases}
	1, & |x-x_R| \le (11\rho + \rho\sht (R))/12; \\
	1, & |x-x_R| \ge (\rho + 11\rho\sht (R))/12; \\
	0, & (5\rho + \rho\sht (R))/6 \le |x-x_R| \le (\rho + 5\rho\sht (R))/6;
\end{cases}
\\
|\nabla \sigma| \le \frac{12}{\rho\sht (R) - \rho}.
\end{gather*}
We claim inequalities
\begin{equation}
||\sigma u_R||_{W_p^1(\Omega_R)} \le ||u_R||_{W_p^1(\Omega_R)} + o_R(1),
\label{cut-offWp-norm}
\end{equation}
\begin{equation}
||\sigma u_R||_{L_q(\Omega_R)} \ge 1 - o_\eps(1)
\label{cut-offLq-norm}
\end{equation}
hold true for all sufficiently large $R$.
\end{lemma}

\begin{rem}
\label{cut-off-rem}
Suppose we have several concentration sequences. In that case we can use Lemma \ref{cut_off_lemma} to produce cut-off functions for each sequence and then multiply them
to obtain a cut-off function which isolates all of the concentration sequences. Note that \eqref{cut-offLq-norm} and \eqref{cut-offWp-norm} are still satisfied.
\end{rem}

\subsection{Concentration theorem}
\label{ssec:Conc_thm}
In this section we consider the functional
\begin{equation}
\tilde{J}[u] = \frac{\ib (|\nabla u|^p + |u|^p)\, dx}{\Big(\ib |u|^q\, dx \Big)^{p/q}}
\label{functionalJ}
\end{equation}
where $1 < p < q < p^*$.

\begin{lemma}
\label{less_than_two_seq_lemma}
Suppose $u_R \in W_p^1(\Omega_R)$ be a sequence of minimizers of functional \eqref{functionalJ}.
Then $u_R$ has no more than one concentration sequence.
\end{lemma}

\begin{proof}
Since $\tilde{J}[tu] = \tilde{J}[u]$ we may assume that $||u_R||_{L_q}=1$.

Suppose there are two sequences $x_R$ and $y_R$ with weights $\lambda_1$ and $\lambda_2$. We are going to show it is more profitable to have only one of them.
Let $\sigma$ be a cut-off function that isolates $x_R$ and $y_R$ (Remark \ref{cut-off-rem}). Let $\sigma_1$ and $\sigma_2$ be the components of $\sigma$
with $x_R \in \supp\, \sigma_1$ and $y_R \in \supp\, \sigma_2$. Set $\sigma_0 = \sigma - \sigma_1 - \sigma_2$. Without loss of generality, assume
$$
\frac{||\sigma_1u_R||_{W_p^1}^p}{||\sigma_1u_R||_{L_q}^q} \ge \frac{||\sigma_2u_R||_{W_p^1}^p}{||\sigma_2u_R||_{L_q}^q}.
$$
We apply Proposition \ref{hump_destruction_lemma} for functions $a=\sigma_0u_R$, $b=\sigma_1u_R$ and $c=\sigma_2u_R$ and obtain function $v_R$ which satisfies
$$
||v_R||_{W_p^1(\Omega_R)} < ||\sigma u_R||_{W_p^1(\Omega_R)} - \mu,
$$
$$
||v_R||^q_{L_q(\Omega_R)} = ||\sigma u_R||^q_{L_q(\Omega_R)} > 1 - o_\eps(1).
$$
The last inequality is by Lemma \ref{cut_off_lemma}. Thus,
$$\tilde{J}[v_R] < \tilde{J}[\sigma u_R] - \mu_1 \le \tilde{J}[u_R](1-\mu_2 + o_\eps(1) + o_R(1))$$
for some $\mu_2>0$, independent of $\eps$ and $R$, which contradicts the minimality of $u_R$.
\end{proof}

\begin{thr}[Concentration theorem]
\label{Concentration_thr}
Suppose $\Omega_R$ be a sequence of Lipschitz domains in $\RN^n$ such that the set of extension
operators from $W_p^1(\Omega_R)$ to $W_p^1(\RN^n)$ is uniformly bounded in norm
(see \cite[Chapter 6, Section 3]{St}). Suppose $u_R \in W_p^1(\Omega_R)$ is a sequence of minimizers of functional \eqref{functionalJ}. Further, suppose that
$\sup\limits_R ||u_R||_{W_p^1} < \infty$. Then $u_R$ has exactly one concentration sequence of weight $1$.
\end{thr}
\begin{proof}
We again assume that $||u_R||_{L_q}=1$.

Firstly, by Lemma \ref{less_than_two_seq_lemma} we have no more than one concentration sequence.
It remains to prove that it exists and has weight $1$.
Let us assume the converse. There are two cases: either there is a concentration sequence $x_R$
with weight $\lambda < 1$ or no sequence at all.
In the first case consider the cut-off function $\sigma$ from Lemma \ref{cut_off_lemma}, let
$\sigma_1$ be the component which isolates $x_R$ and define
$\sigma_0 = \sigma - \sigma_1$ (exactly as in the previous Lemma). In the second case we take
$\sigma_0 \equiv 1$. Now Proposition \ref{point_concentration_lemma} implies
that $\sigma_0u_R$ satisfies the vanishing condition. Moreover, $||\sigma_0u_R||_{L_q}$ tends to
$1-\lambda$ in the first case and is equal to $1$ in the second case.

Let $\Pi_R$ be the extension operator from $W_p^1(\Omega_R)$ to $W_p^1(\RN^n)$. It follows that
$$
||\Pi_R(\sigma_0u_R)||_{W_p^1} \le ||\Pi_R||\sup\limits_R (||u_R||_{W_p^1} + o_R(1)) < C.
$$
In other words, $\Pi_R(\sigma_0u_R)$ is a bounded sequence in $W_p^1(\RN^n)$.
However, inequality
$$
\lim\limits_{R \to \infty}||\Pi_R(\sigma_0u_R)||_{L_q(\RN^n)} \ge \lim\limits_{R \to \infty}
||\sigma_0u_R||_{L_q(\Omega_R)} > 0
$$
contradicts Lemma \ref{no_vanishing_lemma} since the extension operators preserve the
vanishing condition.
\end{proof}

\section{Appendix B: Breather type solutions and center manifold reduction}
\label{ssec:App B}
The variational approach of previous sections provides families of
breather type solutions (see Subsection \ref{ssec:Breather}) with large enough periods.
It is of interest to find a lower bound for these periods.
An approach to this problem was proposed in \cite{AEKLS}. It relies on the bifurcation
theory and an extension of the center manifold theory to elliptic equations in cylindric
domains (see \cite{Kirsch} and more rigorously in \cite{Mielke}). To explain this approach, let us
rewrite the equation \eqref{MainEq} formally as a dynamical system w.r.t. the ``time''\, $y$:
\begin{equation}\label{formal}
u_y = v,\;v_y = -u_{xx} + u - u^3.
\end{equation}
The system has the plane wall type solution $u_0(x)= \pm\sqrt{2}/\cosh x$ which formally is
``the equilibrium state'' of this system. Let us linearize the system at this
solution and study the related spectral problem. We obtain the linear system
which is equivalent to the Sturm-Liouville equation
\[
{\mathcal L} \phi \equiv -\frac{d^2 \phi}{dx^2}+(1-3u_0^2(x))\phi = \lambda^2\phi.
\]
Consider this equation in the space of even in
$x$ functions in $L^2(\mathbb  R)$ (we may do it due to its $x$-parity). The function $-3u_0^2(x)$
(the potential, if one interprets this equation as the Schr\"odinger equation)
is rapidly decaying as $|x|\to \infty.$ Thus the spectrum of the operator ${\mathcal L}-1$
consists of finitely many negative eigenvalues and the continuous spectrum $[0,\infty)$ \cite[Ch. IX]{FA}.
For the equation in question the unique eigenvalue is $\lambda^2 = -3$ with the eigenfunction
$h(x)= c/\cosh^2(x)$ and the continuous spectrum $[1, \infty)$ is separated from the eigenvalue.

For the system linearized at the equilibrium its spectrum consists of the pair
of pure imaginary eigenvalues $\pm i\sqrt{3}$ and two rays of continuous
spectrum $\lambda \ge 1$ and $\lambda \le -1.$ Thus, if the center
manifold theorem from \cite{Mielke} is valid for this case, then in the whole phase space
one gets a smooth local two dimensional center manifold through the equilibrium
which corresponds to eigenvalues $\pm i\sqrt{3}$.
The restriction of the system (\ref{formal}) to this manifold generates a two-dimensional
Hamiltonian system in ``time'' variable $y$ with the equilibrium -- center -- whose neighborhood is
filled with periodic orbits. These orbits provide periodic in $y$ solutions of the initial
elliptic equation. The limit of their periods, as their amplitudes tend to zero, is the frequency
of linearized oscillations $2\pi/\sqrt{3}$.

In \cite{AEKLS} this scheme was presented with the necessary calculations for the related
operators. But the needed estimates on operators that would allow one to apply the center manifold theorem from
\cite{Mielke} were not proved there.

\section{Appendix C: Breather type solutions and homoclinic orbits}
\label{ssec:App C}

To demonstrate another view on breather type solutions, namely, their connection with
homoclinic orbits of some dynamical system,
we use the Fourier expansion in $y$ variable.
Denote by $l$ the period in $y$ of such a solution, $u(x,y)\equiv u(x,y+l)$. We also
assume $u(x,y)\rightarrow 0$ as $x\rightarrow \pm\infty$.

After plugging the Fourier series into the equation we get an infinite system of ODEs for
the Fourier coefficients. Using the Bubnov--Galerkin
procedure, we truncate this system keeping only three modes: zero mode
and two complex conjugated modes $\exp[\pm i 2\pi y/l]$. Thus the anzatz
\begin{equation}
u(x,y)=(U_1(x)-iV_1(x))e^{-i\frac{2\pi}{l}y}+U_0(x)+(U_1(x)+iV_1(x))e^{i\frac{2\pi}{l}y}
\label{2.2}
\end{equation}
gives an approximation solution.
After plugging (\ref{2.2}) into (\ref{MainEq}), projecting on the subspace of related modes and
scaling we arrive at the following Euler-Lagrange type system of second order ODEs:
\begin{equation}
\left\{
\begin{array}{ll}
U''_1-\big(\frac{4\pi^2}{l^2}+1\big)U_1+6U_1U_0^2+3(U^2_1+V^2_1)U_1=0, \\
V''_1-\big(\frac{4\pi^2}{l^2}+1\big)V_1+6V_1U_0^2+3(U^2_1+V^2_1)V_1=0, \\
U''_0-U_0+2U^3_0+6(U^2_1+V^2_1)U_0=0.
\end{array}\right.
\label{2.3}
\end{equation}
It can be transformed to the Hamiltonian form with Hamiltonian
\begin{multline*}
H=\frac{p^2_0+p^2_1+q^2_1}{2}-\frac{U^2_0+\big(\frac{4\pi^2}{l^2}+1\big)(U^2_1+V^2_1)}{2}\\
+\frac{U^4_0}{2}+3U^2_0(U^2_1+V^2_1)+\frac{3}{4}(U^2_1+V^2_1)^2.
\end{multline*}
It is easy to verify that this system has an additional integral $K = p_1V_1-q_1U_1.$

Solutions corresponding to the plane walls belong to the two-dimensional invariant subspace
$U_1 = p_1 =$
$V_1 = q_1 =0,$ and they form two closed homoclinic loops of the saddle equilibrium $U_0 = p_0 =0.$

Breather type solutions we are searching for correspond to homoclinic orbits
of the zero equilibrium. This equilibrium is a saddle, its eigenvalues
are three nonzero real pairs $\pm 1, \pm \sqrt{1+4\pi^2/l^2},\pm \sqrt{1+4\pi^2/l^2}$.
Hence, the equilibrium has three-dimensional smooth stable and unstable invariant manifolds, both of
which lie in the level $H=0$. Homoclinic orbits to the equilibrium belong to the intersection of
these two manifolds, hence they are doubly asymptotic as $x\to \pm\infty$ to the equilibrium.
Since the system has integral $K$ the value of $K$ is also
preserved along the homoclinic orbit. Therefore homoclinic orbits should  belong to the joint
level of two integrals $H=K=0$.

For the Hamiltonian system obtained the following assertion is valid.

\begin{propos} The system has two homoclinic orbits of the zero equilibrium in the invariant 2-plane
$U_1 = p_1 =$ $V_1 = q_1 =0$ and a one parameter family of homoclinic orbits in the
invariant 4-plane $U_0=p_0=0$.
\end{propos}

The proof is done by means of a direct integration using two
integrals. For the one parameter family expressions for $U_1, V_1$ are as follows
$$
U_1=r\cos\theta=\frac{-2\lambda}{e^{\lambda x}+\frac{3}{2}e^{-\lambda x}}\cos\theta,\quad
V_1=r\sin\theta=\frac{-2\lambda}{e^{\lambda x}+\frac{3}{2}e^{-\lambda x}}\sin\theta
$$
where $\lambda^2=\frac{4\pi^2}{l^2}+1$. After that one can find expressions for
all functions $U_0, U_1, V_1$ and construct approximate solutions $U(x,y)$.

\section{Conclusions}

For the elliptic equation \eqref{MainEq} and its generalizations we have constructed nontrivial
solutions of several types: periodic with various lattices of periods,
breather type and radial.  Of course, many questions remain open. It seems that using methods
of \cite{MieZel} it is possible to obtain solutions with finite number
of humps located at different points on the plane with pairwise distances large enough. Also
it would be interesting to prove the existence of solutions which are localized
in both variables, invariant under the rotation by the angle $\pi/2$ and being not
rotationally invariant. Simulations with the equation (\ref{MainEq}) showed their existence
(see Figure 3 in \cite{Lerman2}).

\section{Acknowledgement}

L.M. Lerman acknowledges a partial support from the Russian Foundation of Basic Research
(grant 18-29-10081) and Ministry of Science and Education of Russian Federation
(project 1.3287.2017, target part).

A.I. Nazarov acknowledges a partial support from the Russian Foundation of Basic Research
(grant 17-01-00678).

\end{document}